\documentclass[12pt,leqno,twoside]{amsart}

\usepackage[latin1]{inputenc}
\usepackage[T1]{fontenc}
\usepackage[colorlinks=true, pdfstartview=FitV, linkcolor=blue, citecolor=blue, urlcolor=blue]{hyperref}
\usepackage{amstext,amsmath,amscd,bezier,indentfirst,amsthm,amsgen,enumerate, geometry}
\usepackage[all,knot,arc,import,poly]{xy}
\usepackage{amsfonts,color, soul}  
\usepackage{amssymb}
\usepackage{latexsym}
\usepackage{epsfig}
\usepackage{graphicx}
\usepackage{srcltx}
\usepackage{enumitem, comment}

\newtheorem{theorem}{Theorem}[section]
\newtheorem{corollary}[theorem]{Corollary}
\newtheorem{lemma}[theorem]{Lemma}
\newtheorem{proposition}[theorem]{Proposition}

\theoremstyle{remark}
\newtheorem{remark}[theorem]{\sc Remark}
\theoremstyle{remark}

\theoremstyle{definition}
\newtheorem{definition}[theorem]{Definition}
\theoremstyle{remark}
\newtheorem{example}[theorem]{\sc Example}

\theoremstyle{remark}

\theoremstyle{remark}

\renewcommand{\Box}{\square}    

\newcommand{\cal}{\mathcal}

\newcommand{\aff}{{\mathrm{aff}}}
\newcommand{\Atyp}{{\mathrm{Atyp\hspace{2pt}}}}

\newcommand{\rank}{\mathrm{rank\hspace{2pt}}}

\renewcommand{\int}{{\mathrm{int}}}
\newcommand{\gen}{{\mathrm{gen}}}
\newcommand{\pol}{{\mathrm{pol}}}

\newcommand{\Sing}{{\mathrm{Sing\hspace{1pt} }}}

\newcommand{\im}{{\mathrm{Im\hspace{1pt}}}}

\newcommand{\mult}{{\mathrm{mult}}}

\newcommand{\GL}{{\mathrm{GL}}}
\newcommand{\Mat}{{\mathrm{Mat}}}

\newcommand{\grad}{\mathop{\mathrm{grad}}\nolimits}
\newcommand{\ity}{{\infty}}

\newcommand{\lk}{{\mathbb{C}\mathrm{lk}}}

\newcommand{\rk}{\mathrm{rank\hspace{2pt}}}
\newcommand{\e}{\varepsilon}
\newcommand{\spec}{{\mathrm{spec}}}

\newcommand{\m}{\setminus}
\newcommand{\fin}{\hspace*{\fill}$\Box$\vspace*{2mm}}

\newcommand{\cA}{{\cal A}}
\newcommand{\cB}{{\cal B}}

\newcommand{\cP}{{\cal P}}

\newcommand{\cG}{{\cal G}}
\newcommand{\cH}{{\cal H}}

\newcommand{\cT}{{\cal T}}

\newcommand{\cW}{{\cal W}}

\newcommand{\bC}{{\mathbb C}}

\newcommand{\bP}{{\mathbb P}}

\newcommand{\bZ}{{\mathbb Z}}
\newcommand{\bX}{{\mathbb X}}

\newcommand{\bQ}{{\mathbb Q}}

\begin{document}

\title[Polar degree and vanishing cycles]
 {Polar degree and vanishing cycles}

\author{\sc Dirk Siersma}  

\address{Mathematisch Instituut,     Universiteit Utrecht, Utrecht University, PO
Box 80010, \ 3508 TA Utrecht, The Netherlands.}

\email{D.Siersma@uu.nl}

\author{\sc Mihai Tib\u ar}

\address{Univ. Lille, CNRS, UMR 8524 - Laboratoire Paul Painlev\'e, F-59000 Lille, France.}

\email{mtibar@univ-lille.fr}

\thanks{The authors thank the Mathematisches Forschungsinstitut Obewolfach for supporting this research project through the Research in Pairs program, and acknowledge the support of the Labex CEMPI grant (ANR-11-LABX-0007-01). }

\subjclass[2010]{32S30, 14C17, 32S50, 55R55}


\keywords{singular projective hypersurfaces, polar degree, homaloidal hypersurfaces, nonisolated singularities.}





\begin{abstract}
We prove that the polar degree of  an arbitrarily singular projective hypersurface can be decomposed as a sum of non-negative numbers which quantify local vanishing cycles of two different types.    
This yields lower bounds for the polar degree of any singular projective hypersurface.
 \end{abstract}

%

\maketitle

\setcounter{section}{0}


\section{Introduction}\label{s:intro}
The notion of \emph{polar degree} is primarily defined as the topological degree of the gradient mapping 
\begin{equation}\label{eq:grad}
\grad f : \bP^{n}\m \Sing (V) \to  \check \bP^{n},
\end{equation}
associated to a projective hypersurface $V:= \{ f=0\} \subset \bP^{n}$, for some  homogeneous polynomial $f : \bC^{n+1} \to \bC$.
Thus $\pol(V) := \# (\grad f)^{-1}(l)$ for any generic point $l\in \check \bP^{n}$.

It is known that $\pol(V)$ depends only on $V$ and not on the defining polynomial $f$. This fact was conjectured by Dolgachev \cite{Do}, and proved in \cite{DP},  who considered this invariant under the reason that the gradient mapping \eqref{eq:grad} with $\pol(V) =1$ is a \emph{Cremona transformation}.  The corresponding hypersurfaces $V$ were called \emph{homaloidal}, and Dolgachev classified the projective  plane curves with this property.  The case $\pol(V) =0$ had been studied long ago by Hesse \cite{Hes51, Hes59}, and Gordon and Noether \cite{GN76}; see Example \ref{ex:notcone}. We refer to \cite{Di2} and \cite{Huh} for several remarks about the historical landmarks and for some recent bibliography.

In the beginning of the 2000's, Dimca and Papadima \cite{DP} gave the following topological interpretation:  
\emph{For any projective hypersurface $V$, if $\cH$ is a general hyperplane  with respect to $V$,  then the relative homology $H_{*}(\bP^{n} \m V, (\bP^{n} \m V) \cap \cH)$
 is concentrated in dimension $n$, the  homology $H_{*}(V \m \cH)$
 is concentrated in dimension $n-1$, and one has the equalities: 
\begin{equation}\label{eq:red}
  \pol(V) =  \rk H_{n}(\bP^{n} \m V, (\bP^{n} \m V) \cap \cH) =  \rk H_{n-1}(V \m \cH).
\end{equation}
}

By repeated slicing,  the first equality in formula  \eqref{eq:red} might be expanded as a Cayley-Bacharach type formula with alternating signs, but such a formula is not expected  to provide lower bounds for $\pol(V)$ that one would need for classifying hypersurfaces $V$ with low polar degree. As a breakthrough,  more recently
   June Huh \cite{Huh} obtained positive lower bounds for $\pol(V)$ by using \emph{non-generic hyperplane pencils}, a technique developed in the 2000's \cite{T-conect, T-vcy, Ti-book}
which extends the Lefschetz method to non-generic  pencils having stratified isolated singularities in the base locus. His main result says:

\medskip
\noindent
\textbf{Theorem} \cite[Theorem 2 and its Proof]{Huh}\label{t:pol1mu} 
\emph{Let $V\subset \bP^{n}$ be a hypersurface with isolated singularities.
For any general hyperplane $\cH_{p}$ passing through some singular point $p\in \Sing (V)$, such that $V$ is not a cone of apex $p$, one has:}
\begin{equation}\label{eq:huhkey}
\pol(V) =  \mu^{\left< n-2\right> }_p(V)  + \rk H_{n}(\bP^{n} \m V, (\bP^{n} \m V)\cap \cH_{p})  
\end{equation}

\medskip
Huh's remarkable refinement of Dimca-Papadima's formula \eqref{eq:red} in the particular case of isolated singularities presents $\pol(V)$ as a sum of two non-negative numbers, one of which being the local \emph{Milnor-Teissier number} $\mu^{\left< n-2\right> }_p$ that is well defined whenever $p$ is an isolated singular point of $V$.\footnote{$\mu^{\left< n-2\right> }_p$ denotes the Milnor number at $p$ of the generic hyperplane section through $p$.}  
Then Huh uses the following bound provided by \eqref{eq:huhkey}:
\begin{equation}\label{eq:huhbound}
\pol(V) \ge  \mu^{\left< n-2\right> }_p(V)
\end{equation}
  for showing that there are no  homaloidal hypersurfaces with isolated singularities  besides the smooth quadric and the plane curves found by Dolgachev, and thus confirming a conjecture stated by Dimca and Papadima \cite{DP, Di2}. 
  
 More recently,  the paper  \cite{SST} confirmed the classification list conjectured by Huh \cite[Conjecture 20]{Huh}  of  projective  hypersurfaces  with $\pol(V)=2$ and having isolated singularities. This list consists of  nine plane curves of degrees 3, 4 and 5, and three cubic surfaces. Moreover,  \cite[Theorems  1.4, 4.5 and 4.1]{SST} proved yet  another conjecture of  \cite{Huh} which says that for each fixed polar degree value $\pol (V)$,  the degree  $d$ of $V$  and the number of variables  $n$ have upper bounds.

 \smallskip

Stimulated by the formula \eqref{eq:huhkey}, one may ask the audacious question if there exists a  split-formula for $\pol(V)$ into a sum of  non-negative numbers which may predict and control lower bounds\footnote{As for an upper bound, we have $\pol (V) \le (d-1)^{n}$ for any $V$, which can be easily shown by using the definition \eqref{eq:grad} and the behaviour of the degree of the gradient map under deformations.}
 for $\pol (V)$.  
 
 \bigskip
 
In  our paper we take this challenge and,  considering the general setting of $V$ with any singular locus, we prove that the polar degree  $\pol (V)$ has a split formula into certain non-negative numbers which, remarkably, are all \emph{quantifiers of local vanishing cycles}.
More precisely, Theorem \ref{t:polformula} presents a quantization of $\pol(V)$ as the sum:
\begin{equation}\label{eq:nonneg}
 \pol(V)  = \alpha (V,\cH)   + \beta (V,\cH) ,
\end{equation}
where, as we will define in  \S\ref{s:quantisation},   each of the terms $\alpha (V,\cH)$ and $\beta (V,\cH)$   is a sum of  numbers of local vanishing cycles, but of different types. 

 In order to define the local invariants  which compose  $\alpha (V,\cH)$ and $\beta (V,\cH)$,  all of which depending on a projective hyperplane $\cH\subset \bP^{n}$, we need that this hyperplane $\cH$ is \emph{admissible} with respect to $V$. The set of \emph{admissible hyperplanes}  (see Definition \ref{d:admissible})  includes both the set of generic hyperplane slices  considered in \cite{DP} and the set of non-generic hyperplanes considered  in \cite{Huh}, see Remark \ref{c:Zar}. Roughly, the admissible hyperplanes $\cH$ should have a finite number of isolated tangency points with $V$, in the stratified sense.  Part of these may be \emph{special points}. The set of special points  is a  finite subset of $V$  (Definition \ref{d:special}) defined intrinsically, which extends naturally to any hypersurface $V$ the  notion of  ``isolated singularities'', and is defined in terms of local vanishing cycles. 
 
 The admissibility condition has a second aspect (cf Definition \ref{d:admissible}(ii)) which may be interpreted topologically as the existence of a ``singular polar degree relative to $\cH$''.
  Our new type of  Polar Curve Theorem \ref{t:mainpolar1} insures that the hyperplanes passing through a single special point of $V$ are generically admissible.
 
 
  \smallskip
 Roughly speaking,  the term $\alpha (V,\cH)$ collects the numbers of vanishing cycles of the  isolated stratified singularities  of the slice $V\cap \cH$,  whereas $\beta (V,\cH)$ collects the numbers of the vanishing cycles of the singularities of the $n$-variables polynomial  $P_{\cH} := f_{|\bP^{n}\m \cH}$ which are outside $V$, namely of the  singularities  of $P_{\cH}$ in $\bC^{n}\m V:= \bP^{n}\m (\cH\cup V)$, and of the  singularities  of $P_{\cH}$ ``at infinity'' outside $V$.  All these singularities outside $V$ turn out to be isolated  due to the admissibility conditions (cf Definition \ref{d:admissible}(i) and (ii)).

  \medskip
  
 The change of paradigm in this paper consists in the presentation of the polar degree as the sum of quantifiers of local vanishing cycles. Therefore a key role is played here by the enhanced study of geometric vanishing cycles of polynomial functions done for the proof of Theorem \ref{p:F-concentration}. 

\smallskip
Let us give a brief idea how these local invariants fit together in the above formula \eqref{eq:nonneg} of our Theorem \ref{t:polformula}.
Deleting an admissible  hyperplane $\cH$ yields an affine variety $V \setminus \cH$ 
of affine equation $P_{\cH} = 0$.  It then turns out that the polynomial $P_{\cH} : \bC^{n}\to \bC$ has only isolated singularities  outside $V \setminus \cH$,  both on the affine part $\bP^{n}\m \cH$ and at infinity (Corollary \ref{l:isolated}).   This is possible because we work in the context of singularities at infinity defined in the most general setting of \emph{partial Thom stratifications at infinity}, which avoids the use of the stronger condition of  Whitney (b) regularity,  and in which one can still prove that the vanishing cycles are \emph{localizable} at finitely many points (which are in fact 
 the stratified singularities of a certain proper  extension $\tau$ of   $P_{\cH}$).
It further turns out that the isolated stratified singularities are detectable by the presence of a certain local polar curve, with well-defined local polar multiplicities $\lambda(p,t)$, cf Definition \ref{d:lambdainfty}.

Concerning the term  $\beta(V,\cH)$ of \eqref{eq:nonneg}, Theorem \ref{p:F-concentration} shows  that the reduced homology of  a ``thin tube'' $P_{\cH}^{-1}(D)\subset \bP^{n}\m \cH$ around the fibre $P_{\cH}^{-1}(0) = V \setminus \cH$  is concentrated in dimension $n-1$. It also shows that the top Betti number $b_{n}(P_{\cH}^{-1}(D))$ is the sum of the  polar multiplicities $\lambda(p,t)$  at the isolated $t$-singularities at infinity  outside $V$,  and of the Milnor numbers of the affine isolated singularities of $P_{\cH}$ outside $V \setminus \cH$.

As for the term $\alpha(V,\cH)$ of \eqref{eq:nonneg}, it is the sum of the Milnor-L\^{e} numbers $\alpha_{p}(V,\cH)>0$ of the hyperplane slice $V \cap \cH$ at the isolated non-transversality points $p\in V\cap \cH$, cf \S\ref{ss:sectionalM}. 
The number  $\alpha(V,\cH)$ turns out to be the difference $|\chi(V \setminus \cH) - \chi(V \setminus \cH_{gen})|$, where the Euler characteristic $\chi(V \setminus \cH_{gen})$ is morally the same thing as $\pol(V)$  according to \eqref{eq:pollink}.

We give in \S\ref{ss:computpolar} effective methods for computing the ingredients  $\alpha(V,\cH)$ and $\beta(V,\cH)$. They are used in the explicit computations of several examples, see Examples \ref{ex:betacontrib} and \ref{ex:alphacontrib}.

\smallskip

In the last section \S\ref{s:specV} we discuss lower bounds for the polar degree as a consequence  of the formula  \eqref{eq:alpha} of Theorem \ref{t:polformula}. These are extensions of Huh's  results to non-isolated singularities. In particular they  yield severe obstructions for the existence of homaloidal hypersurface.

 We finally come back to the case $\pol(V) = 0$ first studied by Hesse and Gordon-Noether.  As an application of our results, we show that in this case, if $V$ is not a cone,  then it has no special points. 

\medskip

Our study was also motivated by the rich list of examples of homaloidal hypersurfaces with nonisolated singularities, see e.g. \cite{ES, EKP, CRS, FM, Huh0}.   On the other hand,  the polar degree $\pol(V)$ occurs under several more other significant avatars. The presentation  $\pol(V) = \rk H_{n-1}(V \m \cH)$,
 for a general hyperplane $\cH$, is equivalent to the presentation in terms of the affine complement\footnote{By the Lefschetz slicing method for quasi-projective spaces, one has that the reduced homology $H_{*}(\bC^{n}\m  V)$ is concentrated in dimension $n$, and therefore $\chi(\bC^{n}\m  V) = \chi(\bC^{n}) - \chi(V\m \cH)$.}, 
 as observed in \cite{DP}:
 
 \begin{equation}
 \pol(V) =  \rk H_{n}(\bC^{n}\m V) +  (-1)^{n}
\end{equation}
  where $\bC^{n} = \bP^{n}\m \cH$. The topology of the affine complement  $\bC^{n}\m V$ has been studied notably by Libgober in a series of papers \cite{Li1, Li2, Li3, Li4, LT} etc, and the study was extended to  Alexander modules, e.g. in \cite{Ma1, LM}. In the particular case of an arrangement of hyperplanes $V= V_{\cA}\subset \bP^{n}$, it was noticed in \cite[Cor. 4]{DP} that one also has the presentation: 
  \[   \pol(V) =  \rk H_{n}(\bP^{n}\m V_{\cA}).
  \]
The study of the complement $\bP^{n}\m V_{\cA}$, notably in relation to its combinatorics, is an ample branch of research, with a lot of interesting contributions;  the reader is referred to the relevant literature.  
  The number $\pol(V)$  also appears  in \cite{Ma2} as the middle $L^{2}$ Betti number of the affine complement $\bC^{n}\m V$.

\section{A constrained polar curve theorem}\label{oberwolfach}

Let $f : \bC^{n+1} \to \bC$ be a non-identically zero homogeneous polynomial in fixed variables $x_{0}, x_{1},  \ldots , x_{n}$.  One may assume that $f$  is reduced since the polar degree depends on the reduced structure only, as shown by Dimca and Papadima's \cite{DP}  formula \eqref{eq:red}.
The singular set $\Sing f$ is a cone at the origin $0\in \bC^{n+1}$, i.e. a union of lines passing through the origin, with $\dim 0\le \Sing f \le n-1$.

Let  $\hat l = \sum_{i=0}^{n}a_{i}x_{i}\in \check\bP^{n}$ be a linear form.
\begin{definition}
  One calls
\begin{equation}\label{eq:polarcurve}
 \Gamma(\hat l, f) := \overline{\left\{ x\in \bC^{n+1} \mid \rank 
\begin{bmatrix}
   \frac{\partial f}{\partial x_{0}}(x) &  \frac{\partial f}{\partial x_{1}}(x) & \cdots & \frac{\partial f}{\partial x_{n}}(x) \\
   a_{0} & a_1 & \cdots & a_{n} 
\end{bmatrix} <2 \right\} \m \{ f=0\} }
\end{equation}
the polar locus of $f$ with respect to $\hat l$.
\end{definition}

By definition $\pol(V) := \# (\grad f)^{-1}(\hat l)$ for any general point $\hat  l\in \check \bP^{n}$, thus
  $\pol(V)$ equals the multiplicity  $\mult_{0} \Gamma(\hat l, f)$ at $0\in \bC^{n+1}$ of the polar curve $\Gamma(\hat l, f)$.  Since $f$ is a homogeneous polynomial, this multiplicity is also equal  to the Milnor number of the complex link $\lk_{0}(\{f=0\})$ of the hypersurface germ $(\{f=0\}, 0) \subset (\bC^{n+1}, 0)$, and it is well-known (cf \cite{Le}) that  
$\lk_{0}(\{f=0\})$ is homotopy equivalent to a bouquet of spheres of dimension $n-1$. Therefore $\pol(V)$ is equal to the number of these spheres:
\begin{equation}\label{eq:link}
 \pol(V) = \rk H_{n-1}(\lk_{0}(\{f=0\})).
 \end{equation}

Since $f$ is homogeneous, the local complex link $\lk_{0}(\{f=0\})$ is homeomorphic to the affine set $\{f=0\}\cap \{\hat l=1\} \subset \bC^{n+1}$ for some general $\hat l$, which in turn identifies to $V \m \cH_{\gen}$, where $\cH_{\gen} := \{ \hat l=0\}$ is a general projective hyperplane.
We thus get the equality\footnote{This argument using the polar curve $\Gamma(\hat l, f)$ is different  than the proof developed in \cite{DP} for the corresponding part of the equality \eqref{eq:red}.}:
\begin{equation}\label{eq:pollink}
 \pol(V) = \rk H_{n-1}(V \m \cH_{\gen}).
\end{equation}

In the following we shall identify a hyperplane $\cH\subset \bP^{n}$ to a point $\hat l\in \check \bP^{n}$,
 also viewed as a linear function $\hat l :\bC^{n+1}\to \bC$ (modulo multiplication by a non-zero complex number)  such that $H := \{\hat l = 0\}$ is the associated hyperplane in $\bC^{n+1}$.

\medskip
\noindent
One says that \emph{$V\subset \bP^n$ is a cone of  apex $p$},  for some point $p = [p_{0}; \cdots; p_{n}]\in V$, if the derivatives 
of $f$ satisfy the linear equation $\sum_{i=0}^{n}p_{i}\partial f/\partial x_{i} =0$ over all $\bC^{n+1}$.

Modulo a linear change of coordinates, we can write $p= [1: 0: \cdots : 0]$ and then this condition amounts to: ``the polynomial $f$ does not depend on the coordinate $x_{0}$''. In particular this implies that  $\pol (V) =0$, directly from the definition.  Note that a cone $V$ may have one apex or a linear subspace of apex points in case it is an iterated cone.



\medskip
\noindent
We say that $\cH_{\gen}\subset \bP^{n}$ is a  \emph{generic hyperplane  with respect to $V$} if $\cH_{\gen}$ intersects $V$ transversely in a stratified sense, namely: after endowing $V$ with some Whitney stratification, $\cH_{\gen}$ must be transversal to all strata, in particular avoiding all the 0-dimensional strata. In particular $\cH_{\gen}$  contains no isolated singular points of $V$. It is well-known that the set of generic hyperplanes  with respect to $V$ is a Zariski-open subset of $\check \bP^{n}$.  

We are interested here in 
hyperplanes through some singular point $p\in \Sing (V)$.
Let us  write $\check \bP_{L}^{n-1}$ for the set of hyperplanes $H\subset  \bC^{n+1}$ containing a fixed line $L\subset \bC^{n+1}$  through the origin, and denote by $[L]$ the corresponding point in $\bP^{n}$.

\smallskip

We introduce a new polar curve theorem in such a \emph{non-generic} setting:
\begin{theorem}[Constrained polar curve theorem] \label{t:mainpolar1}\ \\
Let  $f : \bC^{n+1} \to \bC$, $n\ge 2$,  be a  homogeneous polynomial with $\dim \Sing f > 0$.
Let $L\subset \Sing f$ be a singular line such that $V:= \{f=0\}\subset \bP^{n}$ is not a
 cone of apex $[L]$.

  Then there is a Zariski open dense subset $\hat \Omega_{L}\subset \check \bP_{L}^{n-1}$ such that the polar locus $\Gamma (\hat l, f)\subset \bC^{n+1}$ is either a curve for all $\hat l\in\hat \Omega_{L}$, or it is empty for all $\hat l\in\hat \Omega_{L}$.
\end{theorem}

\begin{remark}\label{r:classicalpolarcurve}
 The classical \emph{local polar curve theorem}\footnote{See \cite[Chapter 7]{Ti-book} for details and bibliography.} \cite{HL} says that if the hyperplane $H_{\gen} := \{ \hat l =0\}\subset \bC^{n+1}$ is general,  then the germ at the  origin  $\Gamma_{0} (\hat l, f)\subset \bC^{n+1}$ of the polar locus is either a curve, or it is empty.  Our  Theorem \ref{t:mainpolar1} involves non-general hyperplanes, but on the other hand it applies to homogeneous polynomials only.
\end{remark}

\smallskip
\subsection{Proof of Theorem \ref{t:mainpolar1}}

By a linear change of coordinates we may assume that $L = \bC\langle (1, 0, \ldots , 0)\rangle$. Consider the transversal hyperplane $K :=\{x_{0}=1\}\subset \bC^{n+1}$ at the point $p = (1, 0, \ldots , 0) \in L$.   Let  $\hat l = \sum_{i=1}^{n}a_{i}x_{i}\in \check\bP^{n}$ be a linear form which does not contain  the variable $x_{0}$.

We shall identify $K$ with $\bC^{n}$, the point $p$ with the origin  of $\bC^{n}$, and the restriction $f_{| K}$ with a polynomial $g : \bC^{n}\to \bC$.
We denote $l := \hat l_{|K} =  \sum_{i=1}^{n}a_{i}x_{i}$.  Let

\[ \Gamma(l, g) := \overline{\left\{ x\in \bC^{n} \mid \rank 
\begin{bmatrix}
   \frac{\partial g}{\partial x_{1}}(x)  & \cdots & \frac{\partial g}{\partial x_{n}}(x) \\
   a_{1}  & \cdots & a_{n} 
\end{bmatrix} <2 \right\} \m \Sing g }
\]
be the polar locus of $g$ with respect to $l$.

\begin{lemma} \label{l:nongenpolar}
The polar locus $\Gamma(\hat l, f)$ is the cone over its slice by $K$, and the following equality holds:

\[  K\cap  \Gamma(\hat l, f)  =   \overline{\Gamma(l, g) \cap \left\{ \frac{\partial f}{\partial x_{0}}=0\right\} \m \{ f=0\}}.
\]
\end{lemma}
\begin{proof} The first assertion follows from the remark that $\Gamma(\hat l, f)\subset \bC^{n+1}$ is a homogeneous set and that the hyperplane $K$ is transversal to it. Then the claimed equality follows from the definition \eqref{eq:polarcurve} of the local polar locus $\Gamma(\hat l, f)$, which in our case is a homogeneous set.
\end{proof}

In order to pursue, we need to use certain details of the proof of the classical result, therefore we outline the arguments and refer to \cite{HL, Ti-compo, Ti-book} for the full discussion:
\begin{lemma}[Generic Affine Polar Curve Lemma]  \cite{Ti-compo}, \cite[Theorem 7.1.2]{Ti-book} \label{l:genpolar} \
  \\
Let $g : \bC^{n}\to  \bC$ be a polynomial function.
There is a Zariski open dense subset $\Omega \subset \check \bP^{n-1}$ such that the polar locus $\Gamma(l,g)$ is either a curve for all $l \in \Omega$, or it is empty for all $l \in \Omega$.
\end{lemma}
\begin{proof}
The relative conormal of $g$ is defined as:
\begin{equation}\label{eq:conormal}
 T^{*}_g := \overline{\left\{ (x, \cH) \in (\bC^{n}\m \Sing g) \times \check \bP^{n-1} \mid 
   \cH = T_{x}g^{-1}(g(x))\right\} } \subset \bC^{n}\times \check \bP^{n-1}
\end{equation} 
with projections $\pi_{1}: T^{*}_g \to \bC^{n}$ and $\pi_{2}: T^{*}_g \to \check \bP^{n-1}$.

Before taking the  algebraic closure  in \eqref{eq:conormal} one has by definition a nonsingular open variety of dimension $n$ on which the projection 
$\pi_{1}$ is one-to-one. It follows that $T^{*}_g$ is a (singular) variety of dimension $n$ too.

As $\pi_{2}$ is generically regular by the Morse-Sard theorem, we have two possibilities:\\
 either
\begin{enumerate}
\item $\dim \im \pi_{2} < n-1$, thus $\pi_{2}$ is totally singular. In this case the generic polar locus is empty, i.e. we may take as $\Omega$ the interior of the complement of $\im \pi_{2}$,\\
\hspace*{-1.3cm} or
\item $\dim \im \pi_{2} = n-1$.
\end{enumerate}
In case (b), $\im \pi_{2}$ contains a Zariski open dense subset  $\Omega \subset\check\bP^{n-1}$ over which $\pi_{2}$ is a submersion, more precisely the restriction:
\begin{equation}\label{eq:G} 
 {\pi_{2}}_{|G} : G :=    \pi_{2}^{-1}(\Omega)  \cap [(\bC^{n}\m \Sing g)\times \check \bP^{n-1}] \to \Omega \subset \check\bP^{n-1}
\end{equation}
is a submersion. Consequently its fibre is a complex manifold of dimension 1, hence a curve. Since $\Gamma(l,g) = \overline{\pi_{1}(({\pi_{2}}_{|G})^{-1}(l))}$ for $l\in \Omega$,
this proves our lemma.
\end{proof}

Consider now the image  by $\pi_{1}$ of the smooth algebraic set $G$ of dimension $n$ defined in \eqref{eq:G}. 
Since $\pi_{1}$ is generically one-to-one on it, the image $\pi_{1}(G)$ contains an open 
subset of $\bC^{n}$, hence we  have the following consequence of the above proof:
\begin{corollary}\label{c:polaropen}
Assume that $\dim \im \pi_{2} = n-1$. Let $l_{0}\in \Omega \subset \check \bP^{n-1}$ and let $B_{\e}(l_{0}) \subset \Omega \subset \im \pi_{2}$ be a small enough ball centred at $l_{0}\in \check\bP^{n-1}$.
Then the image $\pi_{1}( {\pi_{2}}_{|G}^{-1}(B_{\e}(l_{0})))\subset \bC^{n}$ contains an open  tubular neighborhood of the curve $\pi_{1}( {\pi_{2}}_{|G}^{-1}(l_{0}))  \subset \Gamma(l_{0}, g)$ in  $\bC^{n}$.
\fin
\end{corollary}

\subsection{End of the proof of Theorem \ref{t:mainpolar1}}\ \\
 The idea is that a linear function $\hat l\in\hat \Omega_{L}$ should be a sufficiently general linear function
 $l = \hat l_{|K}$ on the hyperplane $K\subset \bC^{n}$.  This may be realized in all details, as follows.
 
By applying  Lemma  \ref{l:genpolar} to the restriction $g = f_{| K}$, which is a polynomial function of $n$ variables, we obtain the existence of a Zariski open dense subset $\Omega \subset \check \bP^{n-1}$  such that $\Gamma(l,g)\subset \bC^{n}$ is either a curve for all $l \in \Omega$, or empty for all $l \in \Omega$.  

In case (a) of the proof of Lemma \ref{l:genpolar}, the polar locus $\Gamma(l,g)$  is empty for all $l\in \Omega$. By Lemma \ref{l:nongenpolar} it then follows that 
the polar locus $\Gamma (\hat l, f)$ is empty too,  for any $\hat l$ in the Zariski open dense subset $\hat \Omega_{L}:= \Omega \subset \check \bP_{L}^{n-1}$,
modulo the identification  $\hat l (x_{0}, x_{1}, \ldots , x_{n}) = l( x_{1}, \ldots , x_{n})$ due to the definition of $\hat \Omega_{L}$. This is the trivial case.\footnote{Let us remark that the polar curve $\Gamma(l,g)$ is not empty whenever the polynomial $g$ has at least one \emph{isolated} singularity, since the generic local polar curve at an isolated hypersurface singularity is not empty. More generally, this holds at \emph{special points} of $V$, see Definition \ref{d:special} and related results in \S\ref{s:specV}.}

\medskip

In  case (b) of the proof of Lemma \ref{l:genpolar},   the polar locus $\Gamma(l,g)$  is a curve for all $l\in \Omega$. 
Then let us consider  the homogeneous set $R := \{\frac{\partial f}{\partial x_{0}} =0\}\subset \bC^{n+1}$ which occurs in Lemma \ref{l:nongenpolar}, and remark  that it contains the line $L$.
There are two possibilities: either $R$ is the entire space $\bC^{n+1}$, which is equivalent to the fact that $f$ does not depend on $x_{0}$, and which is excluded by the hypothesis that $V$ is not a cone of apex $[L]$,
or $R$ is a homogeneous hypersurface germ, and thus  of pure dimension $n$. 

In this latter case, it then follows that the intersection $R\cap K \subset K =\bC^{n}$ is of dimension $n-1$,  recalling that the affine hyperplane $K:= \{ x_{0} =1\}\subset \bC^{n+1}$ is transversal to the $\bC^{*}$-orbits of the points of $R$, and that we have identified $K$ with $\bC^{n}$.
 
By Corollary \ref{c:polaropen}, for any $l_{0}\in \Omega$, the space $\pi_{1}( {\pi_{2}}_{|G}^{-1}(B_{\e}(l_{0}))) \subset\bC^{n}$ contains an open set of dimension $n$ and it is by definition the union of  curves $\bigcup_{l\in B_{\e}(l_{0})}\Gamma (l, g)\m \Sing g$. Therefore $\pi_{1}( {\pi_{2}}_{|G}^{-1}(B_{\e}(l_{0})))$ contains an open  tubular  neighborhood of the curve $\Gamma(l_{0}, g)\m \Sing g$ in  $\bC^{n}$, thus contains an open tubular neighborhood of each irreducible component of the  curve $\Gamma(l_{0}, g)\m \Sing g$, as these components are finitely many.

 Continuing the reasoning, since $R\cap K$ has dimension  $n-1$ in $K= \bC^{n}$, the subset: 
 \[  \{ l\in \check \bP_{L}^{n-1} \mid 
\mbox{ some irreducible component of the  curve } \Gamma(l, g)\m \Sing g \mbox{  is contained in } R\cap K \} \]
 is algebraic of dimension $\le n-2$. Its complement is Zariski open, and we denote it by $\hat \Omega_{L}\subset \Omega \subset \check \bP_{L}^{n-1}$. Therefore  the intersection $\Gamma(l, g)\cap R\subset K$  is of dimension zero for any $l\in \hat\Omega_{L}$.

 By Lemma \ref {l:nongenpolar}, it then follows that  the polar set $\Gamma(\hat l, f)$ is either a curve for all $l\in \hat\Omega_{L} \subset \Omega \subset \check \bP^{n-1}$, or it is empty.
 
 This ends the proof of Theorem \ref{t:mainpolar1}. \fin
 
\begin{remark}\label{r:cone2}
By examining the impact of the ``non-conical''  condition in the above proof, we can deduce what is the  complementary situation with respect to the statement of  Theorem \ref{t:mainpolar1}. Namely:

\emph{If  $V$ is  a cone of apex  $[L]$ then the polar set $\Gamma(\hat l, f)$ is either of pure dimension  $2$ for all
$l\in \hat\Omega_{L} \subset \Omega \subset \check \bP^{n-1}$, or it is empty for all
$l\in \hat\Omega_{L} \subset \Omega \subset \check \bP^{n-1}$.}

Indeed, in case $V$ is  a cone of apex  $[L]$,  our polar  set $\Gamma(\hat l, f)$ is the cone over the generic polar locus  $\Gamma(l, g)$ in the slice $K$, and 
in this slice, the generic polar locus is either a curve or empty, which yields the precise dimensions in the above statement.
\end{remark}

\begin{remark}\label{r:pol=0-noncone}
 The  projective hypersurfaces $V$ which are not cones, but have polar degree $\pol (V) =0$, are particularly interesting and of course related  to the original studies by Hesse, Gordon and Noether referred to in the Introduction.  In this case one can have a dimension 1 generic polar locus  $\Gamma(l, g) \subset K = \bC^{n}$, whereas 
 the affine polar curve $\Gamma(\hat l, f)\subset \bC^{n+1}$ is empty. See Example \ref{ex:notcone}.
 \end{remark} 


\bigskip

\section{Vanishing cycles of polynomial functions}\label{ss:singinfty}

The proof of our  polar degree formula of Theorem \ref{t:polformula} relies on a key result, Theorem \ref{t:finiteThom}, that  we prove in this section. This is based on geometric vanishing cycles of polynomial functions. We use the concept of ``partial Thom stratifications'' (cf \cite{Ti-book})  in order to localize those vanishing cycles which vanish asymptotically.  

\smallskip
Let $P: \bC^n \to \bC$ be a non-constant polynomial function of degree $d$. It is well-known that there is a finite subset of 
the target $\bC$ such that $P$ induces a locally trivial fibration over its complement. The minimal such subset is called  the \emph{set of atypical values} $\Atyp P \subset \bC$.

In our case we consider the polynomial  $P(x_{0}, \ldots, x_{n-1}) := f(x_{0}, \ldots, x_{n-1}, 1)$, and we say that $\tilde P = f$ is the homogenized of $P$ of degree $d$ by the coordinate $x_{n}$

Let
 $\bX := \{f(x_{0}, \ldots, x_{n}) -tx_{n}^{d}=0\} = \{\tilde P-tx_{n}^{d}=0\}  \subset \bP^{n}\times \bC$.
Let   $\tau : \bX \to \bC$ be  the projection on the second factor,  and  let us denote by  $\bX_{t} := \tau^{-1}(t)$ its  fibres.

The set $\bX$ is precisely the closure in $\bP^n \times \bC$ of the graph $\{ (x,t)\in \bC^n\times \bC \mid P(x) =t\}$ of $P$ and    $\bX^{\ity} := \bX \cap (\cH\times \bC) = (V\cap \cH)\times \bC$ is the divisor at infinity, where $\cH := H/\bC^{*}$ and $H := \{ x_{n}=0\} \subset \bC^{n+1}$. One may then identify $\bC^n$ with $\bX \setminus \bX^\ity$ via the canonical map $x \mapsto ([x:1], P(x))$. In particular, $\tau$ is a proper  extension of $P$. The hypersurface  $\bX\subset \bP^n \times \bC$ is covered by affine charts $U\times \bC$, where $U\subset \bP^{n}$ is some affine chart of $\bP^{n}$.

\medskip
The non-isolated singular locus of the fibre  $\bX_{0} =V$ (if there is any) intersects the hyperplane at infinity $\cH$. 
Here we are  interested in another type of singularities, the so-called  \emph{singularities at infinity} of the fibres $\bX_{t}$ for $t\not=0$. 

\smallskip
\subsection{Partial Thom stratifications.} \label{ss:thompartial}

We use the following result, as a particular case of the more general result in reference: 

\begin{theorem}\cite{BMM}\footnote{See also  \cite{Ti-compo}, \cite[Theorem A1.1.7]{Ti-book}.}\label{t:**}.
 Let $X\subset \bC^{N}$ be a complex space endowed with a Whitney stratification,  and let $h$ be a holomorphic function  on $X$  such that its zero locus  is a union of strata. Then this Whitney stratification is also a Thom a$_{h}$-regular stratification of $h^{-1}(0)$.
 \fin
\end{theorem}

 The singular locus $\Sing(\bX)$ of the set $\bX$ is included in the divisor at infinity $\bX^{\ity}$ and therefore one may endow $\bX$ with a Whitney stratification such that $\bX^{\ity}$ is a union of strata.
By Theorem \ref{t:**} applied to $h =x_{n}$, this Whitney stratification of $\bX$ is Thom regular at $\bX^{\ity}$ with respect to the levels  of $x_{n}$, in any local chart.  This is a particular case of a \emph{partial Thom stratification at infinity} like introduced in \cite{Ti-compo}, \cite[Def. 2.1]{Ti-equi},  see also \cite[Appendix 1.1]{Ti-book}, \cite{DRT},  \cite{ART}. 
\begin{definition}(Partial Thom  stratification at infinity of $P$, \cite{Ti-compo},  \cite[Def. 2.1]{Ti-equi}, \cite[Def. 9.1.3]{Ti-book}.)
A locally finite stratification of $\bX^{\ity}$ such that each stratum is Thom (a$_{x_{n}}$)-regular with respect to the smooth stratum $\bX\m \bX^{\ity}$ is called a  \emph{$\partial$-Thom  stratification at infinity}.  This is independent on the affine chart where one considers the function $x_{n}$ and its levels, cf \cite[Theorem 3.6]{Ti-compo}.
\end{definition}
The advantage in working with partial Thom stratifications is that one does not require the more involved Whitney regularity, whereas one may still prove fibration theorems, cf \cite{Ti-compo}  \cite[Appendix 1]{Ti-book}.

\begin{definition}\label{d:t-sing}($t$-singularities at infinity.) \cite{ST-singinf}, \cite{Ti-compo}, \cite[Definition 1.2.10]{Ti-book}.
Let $\cG$ be a $\partial$-Thom  stratification at infinity of $\bX$, and let $\eta \in \bX^{\ity}$.  If the map $\tau : \bX \to \bC$ is transversal to the stratification $\cG$ at $\eta$ then we say that $P$ is \emph{$t$-regular at infinity} at this point.  If non-transversality occurs instead, then we say that $P$ has a \emph{$t$-singularity at infinity} at $\eta$.

We say that $\eta$ is an \emph{isolated $t$-singularity} at infinity of $P$ if the map $\tau : \bX \to \bC$ has an isolated non-transversality at $\eta$ with respect to  $\cG$, and if moreover the map $\tau$ has no other singularity on $\bX \m \bX^{\ity}$ in the neighborhood of $\eta$.
\end{definition}

\subsection{The isolated non-transversality condition, and Thom regularity.}\label{ss:non-transv}\label{ss:Thomreginf}

We fix a Whitney stratification $\cal W$ of the homogeneous hypersurface $\{f=0\}\subset \bC^{n+1}$. We may and will  assume that $\cal W$ is a homogeneous stratification, and recall that the strata of the coarsest Whitney stratification are $\bC^{*}$-invariant.

We consider the origin of $\bC^{n+1}$ as a point-stratum and denote by $\cal W^{*}$ the union of all strata different from $\{0\}$. Then the projectivized stratification  $\bP \cal W := \cal W^{*}/\bC^{*}$ yields a Whitney stratification of the projective hypersurface $V$.
Let us also recall  that the Whitney stratification $\bP \cW$ of the hypersurface $V\subset \bP^{n}$  depends on the \emph{reduced} structure of $V$ only.

Let  $H:= \{\hat l=0\}$ be a hyperplane through $0\in \bC^{n+1}$, together with its associated projective  hyperplane $\cH := H/\bC^{*} \subset \bP^{n}$ which satisfies the following property with respect to the stratification $\bP \cW$ of $V$ that we will call  \emph{isolated non-transversality}:


\begin{equation}\label{eq:star}
 \cH \mbox{ \textit{is transversal to all strata of} } \bP\cal W \mbox{ \textit{except at finitely many points.} }
\end{equation}

\medskip

 Let us denote by $\Sing_{\bP \cW}(V\cap \cH) = \{p_{i}\}_{i=0}^{r} \subset \bP^{n}$ the set of points of non-transversality in \eqref{eq:star}, and by $L_{i}\subset \bC^{n+1}$ the line which corresponds to $p_{i}$.

By the transversality property \eqref{eq:star}, the induced stratification $H\cap \cal W^{*}$ on $H\cap \{f=0\}\m \bigcup_{0}^{r} L_{i}$ is  Whitney regular.  It extends to a homogeneous  Whitney regular stratification of $H\cap \{f=0\}$ where  
the lines $L_{i}$ are the Whitney strata of dimension 1. We denote by  $\cW_{H}$ this Whitney stratification, and by $\bP \cW_{H} := \cW_{H}/\bC^{*}$ the corresponding Whitney stratification of $V\cap \cH$. Its 0-dimensional strata are precisely the points $\{p_{i}\}_{i=0}^{r}$.


We consider the map $(\hat l,f)  : (\bC^{n+1},0) \to (\bC^{2}, 0)$.  Modulo  a linear change of coordinates, one may assume that $\hat l = x_{n}$. We thus consider the affine map $(x_{n},f): \bC^{n+1} \to \bC^{2}$ and the Thom regularity condition of this map  at its central fibre\footnote{It is the inverse image of the origin that is called here ``central fibre''.}. One defines the \emph{Thom non-regularity locus}\footnote{See  \cite{PT}.} as the set of points of the central fibre of a map where the Thom regularity condition fails with respect to some stratification defined  in the preamble, which in our case is $\cW_{H}$.

\begin{lemma}\label{l:thom}
If the condition \eqref{eq:star} holds, then the stratification $\cW_{H}$ is a Thom $\mathrm{a}_{(\hat l, f)}$-regular stratification of the central fibre $H\cap f^{-1}(0) \subset \bC^{n+1}$ at all points outside the union of lines  $\bigcup_{0}^{r} L_{i}$. 
In other words, the Thom $\mathrm{a}_{(x_{n}, f)}$ non-regularity locus of $(x_{n},f)$ with respect to $\cW_{H}$  is at most $\cup_{i=0}^{r} L_{i}$.
\end{lemma} 

\begin{proof}
Point-strata of $\bP \cW$ may be either outside $\cH$, or inside it, and then these are among the non-transversality points $p_{i}$.  Excluding the union of lines $\cup_{i=0}^{r} L_{i}$, we thus consider some point $y_{0}\in H\cap W \subset H\cap f^{-1}(0)$, where $W\in \cW$ is a stratum of dimension $\ge 2$. By the transversality condition \eqref{eq:star}  we have $H\pitchfork_{y_{0}} W$, and so $\dim H\cap W \ge 1$.

Let $y_{k}\in \bC^{n+1}\m f^{-1}(0)$ be some sequence of points $y_{k}\to y_{0}$. Theorem \ref{t:**}  applied to $f$ and to the Whitney stratification $\cW_{H}$ implies that the limit of tangent hyperplanes
$\lim_{y\to y_{0}} T_{y}f^{-1}(f(y))$, whenever it exists,  contains $T_{y_{0}}W$. Since $H$ is transversal to $T_{y_{0}}W$, it  follows that the tangent space of the fibre of $(\hat l,f)$ passing through some point $y$ close enough to $y_{0}$ is the transversal intersection of two tangent spaces:  the tangent space to the fibre of $\hat l$,  and the tangent space to the fibre of  $f$,
since  this converges to the intersection $H\cap \lim_{y\to y_{0}} T_{y}f^{-1}(f(y))$ which contains $H \cap T_{y_{0}}W$. 

Let us also consider a sequence $y_{k}\to y_{0}$ with $y_{k}$ belonging to some stratum of the Whitney stratification $\cW$ of $\{f=0\}$. 
We then apply Theorem \ref{t:**}  to the function $\hat l = x_{n}$ on the space  $\{f=0\}$ and we conclude that the Thom regularity holds at the point $y_{0}$.

This shows that $H \cap W$ is a Thom regular stratum for the map $(\hat l,f)$. 
 \end{proof}

We shall call  $\cH = \{ x_{n} =0\}$ the \emph{hyperplane  at infinity} for the polynomial $P_{\cH} := f(x_{0}, \ldots , x_{n-1}, 1)$, such that  $P_{\cH}: \bC^{n}\to \bC$ is a polynomial of degree $d$.
Thus $f = \tilde P_{\cH}$ is the homogenized of degree $d$ of $P_{\cH}$ by the variable $x_{n}$.  Writing
\[ f(x_{0}, \ldots , x_{n}) = f_{d}( x_{0}, \ldots , x_{n-1}) + x_{n}f_{d-1}( x_{0}, \ldots , x_{n-1}) + \cdots ,
\]
we have $\cH\cap \Sing (V) = \{\partial f_{d}=0, f_{d-1}=0\}\subset \cH \simeq \bP^{n-1}$, and thus  $\Sing(\bX) = (\cH\cap \Sing (V)) \times \bC \subset \bX^{\ity}$.  
We therefore consider the product stratification $\bP \cW_{H}\times \bC$ of $\bX^{\ity}$,  the strata of which are the products $\hat W\times \bC$ for all $\hat W\in \bP \cW_{H}$.

We are now ready to state our key result on the $\partial$-Thom stratification at infinity introduced in \S\ref{ss:thompartial}, under the condition \eqref{eq:star} given in  \S\ref{ss:non-transv}, also expressed in terms of the $t$-regularity, cf Definition \ref{d:t-sing}:

\begin{theorem}\label{t:finiteThom}
If the isolated non-transversality condition \eqref{eq:star} holds,  then the product stratification $\bP \cW_{H}\times \bC$ of $\bX^{\ity}$  is a partial Thom stratification at infinity of $P_{\cH}$, possibly except at finitely many points on the lines $\{ p_{i}\}_{i=0}^{r} \times \bC \subset \bX^{\ity}$.

In particular,  the polynomial $P_{\cH}$ is $t$-regular at infinity at all points of $\bX^{\ity}$, except of finitely many points on the lines $\{ p_{i}\}_{i=0}^{r} \times \bC \subset \bX^{\ity}$, $i=0, \ldots , r$.
\end{theorem} 
 \begin{proof} 
  Let  $(\hat x, t_{0}) \in \bX^{\ity}$, such that $\hat x \in \cH\cap V \m \{ p_{i}\}_{i=0}^{r}$.  Let $\hat W \in\bP \cW_{H}$ be the stratum containing $\hat x$, where $\hat W := W/\bC^{*}$ denotes the projective image of a stratum $W\in \cW^{*}_{H} := \cW \cap H$.  The stratification $\bP \cW_{H}$ can have 0-dimensional strata other than\footnote{Then these are by definition points on some 1-dimensional strata of $\bP \cW$ to which $\cH$ is transversal.} the points $p_{i}$, but in any case we have $\dim W \ge 1$.


We recall that $H := \{x_{n}=0\}$. We may and shall assume, possibly after a linear change of coordinates, that $\hat x = [1; 0; \ldots ; 0]$.
 
In the   affine chart $\bC^{n}\times \bC$ defined by $x_{0}\not= 0$, the hyperplane slice $\bX \cap  \{ x_{n}= s\}$
is of pure dimension $n-1$ and is defined as the intersection of two smooth spaces:
\begin{equation}\label{eq:fibretot}
 \{ f(1, x_{1}, \ldots , x_{n-1}, x_{n})  - tx_{n}^{d} =0 \} \cap \{x_{n} = s \}.
 \end{equation}
 
 We claim that the above intersection \eqref{eq:fibretot}  is transversal and that, if its
 tangent space along a sequence of points $(y,t)\to (\hat x, t_{0})\in \hat W\times \{t_{0}\} \subset \bX^{\ity}$
 converges to a limit $T$  (which depends on the chosen sequence of points), then $T\supset T_{\hat x}\hat W \times \bC$.

\medskip   
The tangent space at some regular point $(y,t)$ of the hypersurface defined by \eqref{eq:fibretot} 
in  the hyperplane $\{y_{n} = s\}$  has as normal vector:
\[n(y,s) := \left( \frac{\partial f}{\partial x_{1}}(y), \ldots, \frac{\partial f}{\partial x_{n-1}}(y) , - s^{d}\right) \in  \{y_{n}=s\}\times \bC =\bC^{n-1}\times \bC .
\]

\medskip

The first part $n'(y) := (\frac{\partial f}{\partial x_{1}}, \ldots, \frac{\partial f}{\partial x_{n-1}})(y)$ of the vector $n(y)$ represents the normal direction at $y$ to the slice by the hyperplane $\{x_{0}=1\}$ of the corresponding fibre $(x_{n}, f)^{-1}(s,ts^{d})$ of the map $(x_{n}, f)$ considered in Lemma \ref{l:thom}. Then Lemma \ref{l:thom} shows that, whenever the following limit of tangent spaces exists:
\[ \lim_{(y,t)\to ((0,\ldots ,0),t_{0})}T_{(y,t)}(x_{n}, f)^{-1}(s,ts^{d}),\] 
it must contain the tangent space $T_{(1, 0, \ldots, 0)}W$, which is of positive dimension. Since the slice  $\{x_{0}=1\}$ is transversal to the stratum $W$, it is also transversal to the fibre $(x_{n}, f)^{-1}(s,ts^{d})$ at the point $y$, for all  $(y,t)\to ((0,\ldots ,0),t_{0})$.

This transversality shows that the derivatives  $\frac{\partial f}{\partial x_{1}}, \ldots, \frac{\partial f}{\partial x_{n-1}}$ are non-identically zero along the chosen sequence of points $(y,t)$, and that the following limit direction is well-defined:
\begin{equation}\label{eq:limitvect}
\overrightarrow{n} =\lim_{y\to 0} \frac{\left( \frac{\partial f}{\partial x_{1}}; \cdots; \frac{\partial f}{\partial x_{n-1}}\right)(y)}{ \left\|   \left( \frac{\partial f}{\partial x_{1}}; \cdots; \frac{\partial f}{\partial x_{n-1}}\right)(y) \right\|}
\end{equation}
 and is normal to the slice $T_{(1, 0, \ldots, 0)}W \cap \{x_{0}=1\}$, since this is a transversal slice of the tangent space $T_{(1, 0, \ldots, 0)}W$ and has positive dimension. Note that in case $\dim W=1$ the slice is one point only (thus we do not have a normal vector anymore). In the general case $\dim W\ge1$ we have shown that the limit direction $\overrightarrow{n}$ exists. This is the only thing that we need, since now, 
 by identifying the general slice $W \cap \{x_{0}=1\}$ of the homogeneous stratum $W$ with its projectivization $\hat W$, this proves that the vector $(\overrightarrow{n}, 0)$ is normal to the stratum $T_{[1; 0; \ldots; 0]}\hat W \times \bC$, which finishes the proof of our claim.

\medskip
To complete the proof of the first assertion of our theorem, it only remains to check the Thom (a$_{x_{n}}$)-regularity along  the  lines $\{ p_{i}\} \times \bC \subset \bX^{\ity}$, $i=0, \ldots , r$.  We claim 
that all the points of the lines $\{ p_{i}\} \times \bC \subset \bX^{\ity}$, $i=0, \ldots , r$ are $t$-regular except of finitely many.   

Indeed, our stratification  $\bP \cW_{H} \times \bC$ may be refined to a Whitney regular stratification of $\bX$ such that $\Sing(\bX)$ is a union of strata. Since it is semi-algebraic, this stratification has finitely many strata.  By Theorem \ref{t:**}, this is also a Thom (a$_{x_{n}}$)-regular stratification at infinity. Now any line $\{ p_{i}\} \times \bC$ may intersect only finitely many strata of this Thom-Whitney stratification, thus our claim is proved.

The second assertion of the theorem is a consequence of the first, since the projection $\tau : \bX \to \bC$ is transversal to the
strata of dimension $\ge 1$ of $\bP \cW_{H} \times \bC$. The only non-transversality points are therefore the 0-dimensional strata of $\bP \cW_{H} \times \bC$, and we have shown that they are included in the lines  $\sqcup_{i=0}^{r}\{ p_{i}\} \times \bC$, and that they are finitely many.
\end{proof}

\begin{remark}\label{r:detectlambda}
  Under the supplementary condition  ``$H$ is admissible'' which will be introduced below, the finitely set of points $\{(p_{i}, t_{i_{j}})\}_{i,j}$ for $t_{i_{j}} \not= 0$  will turn out to be  \emph{isolated $t$-singularities}.
  In this case they can be precisely detected, and we refer to  the criterion \eqref{eq:lambda(p,t)}, as well as to \S\ref{ss:computpolar}.
\end{remark}
\


\section{Admissible hyperplanes, and vanishing cycles outside the central fibre}
  
The above Theorem \ref{t:finiteThom} does not tell anything about the affine singularities of the 
  polynomial $P_{\cH}$. This can have isolated or nonisolated singularities in a finite number of fibres.  We outline  
  a  class of   polynomials $P_{\cH}$, which translates into a class of hyperplanes $\cH$, such that  the homological vanishing cycles which live outside the fibre   $P_{\cH}^{-1}(0)$ are concentrated in the top degree.
    
\subsection{Admissible hyperplanes and $\cT_{0}$-type polynomials.}
Using the notations of \S\ref{ss:singinfty}, we introduce a class of polynomials which extends that of $\cT$-type of \cite{ST-gendefo}, see also \cite[Def. 2.2.2]{Ti-book}:

\begin{definition}\label{d:T0type}
We say that the polynomial $P: \bC^{n}\to \bC$ is of \emph{$\cT_{0}$-type}  if the following two conditions are both satisfied:

\noindent
(a)  $P$ has only isolated affine singularities outside the fibre $P^{-1}(0)$, and

\noindent
(b)  $\tau$  has only isolated $t$-singularities\footnote{Cf Definition \ref{d:t-sing}.} at infinity outside $\overline{P^{-1}(0)}$.
\end{definition}

As before, let  $f: \bC^{n+1} \to \bC$ be some non-constant homogeneous function, and $V := \{f=0\}\subset \bP^{n}$ is endowed with a Whitney stratification $\bP \cW$. Let also $\hat l: \bC^{n+1} \to \bC$ be a linear function defining a hyperplane $H\in \bC^{n+1}$ and let $\cH\subset  \bP^{n}$ denote its corresponding projective hyperplane.   
\begin{definition}\label{d:admissible}
We say that the affine hyperplane $H\subset  \bC^{n+1}$ through 0 (or that the projective  hyperplane $\cH\subset  \bP^{n}$)  is \emph{admissible for $f$} if:
\begin{itemize}
\item[(i)] the isolated non-transversality condition \eqref{eq:star} holds, namely: $\cH$ is transversal to all strata of $\bP\cal W$ except at finitely many points.
\item[(ii)] the polar locus $\Gamma (\hat l,f) \subset \bC^{n+1}$ is either of dimension 1, or it is empty. 
 \end{itemize}
A hyperplane $\cH$ which is admissible for $f$  and contains a certain  point $p\in V$ will be called \emph{admissible for $f$ at $p$}. 
\end{definition}

\medskip

The above definition of ``admissible hyperplanes'' takes into account all the singularities of the slice $V \cap \cH$, hence it is a global condition.  
The set of admissible hyperplanes contains by definition the set of  \emph{generic hyperplanes $\cH$ relative to $V$}, namely hyperplanes which are transversal to  all strata of the stratification $\bP \cW$ of $V$, since in this case the non-transversality locus is empty. It also turns out that the polar locus $\Gamma(\hat l, f)$ is 1-dimensional or empty for generic hyperplanes $\cH$, thus condition (ii) is fulfilled; we remind that in this case the multiplicity of $\Gamma(\hat l, f)$ is precisely the polar degree of $V$. 

The condition \eqref{eq:star} says that admissible hyperplanes $\cH$ can be non-generic. We are actually interested in admissible hyperplanes containing as many as possible non-transversality points, because it will turn out in  \S\ref{s:quantisation}  that  every such point contributes with a non-negative term in the sums \eqref{eq:alpha}, and thus we get better bound results for the polar degree. Let us notice that condition (ii) tells that the multiplicity of $\Gamma(\hat l, f)$ can be viewed as a singular polar degree of $V$ relative to the admissible $\cH$.

\medskip

The next remark tells that hyperplanes which are admissible for $f$ at some singular point $p\in V$ are \emph{generic} among all hyperplanes containing $p$.

\begin{remark}\label{c:Zar}
 Let $V := \{f=0\}\subset \bP^{n}$ and let $p\in \Sing V$ such that $V$ is not a cone of apex $p$. 
 Then the set of admissible hyperplanes for $f$ at $p$ contains  a Zariski-open subset of the set $\check \bP_{\{p\}}^{n}\simeq \check \bP^{n-1}$ of hyperplanes passing through $p$. 
 
 Indeed, let $L\subset \bC^{n+1}$ be the affine line through 0 which corresponds to the point $p\in \Sing V$.
 By our Polar Curve Theorem \ref{t:mainpolar1}, there is a  Zariski-open dense set $\Omega_{L}\subset \check \bP^{n-1}$ of hyperplanes $H\subset \bP^{n}$ satisfying  condition (ii) of admissibility.  The condition  \eqref{eq:star} is also a Zariski-open condition in $\check \bP_{\{p\}}^{n}$.  The intersection of these two Zariski-open dense subsets of $\check \bP^{n-1}$ is a solution to our claim.
 See also Corollary \ref{c:coneornotcone}.

Whatsoever, 
 the existence  of admissible hyperplanes with two isolated tangency points  is not insured, in the sense that the set of such hyperplanes is not a Zariski-open subset anymore.
 This can be checked in the lists of $\pol(V) =1$ or $\pol(V) =2$, cf \cite{Huh} and \cite{SST}.
  In case of the plane curve defined by $x_{0}x_{1}x_{2}=0$, one has 3 singular points of type $A_{1}$
  but there is no admissible line in $\bP^{2}$ passing through two of them because such a line is contained in the curve.
     In case of the normal surface $x_{0}x_{1}x_{2} +x_{3}^{3}=0$ with three singular points $A_{2}$
    and $\pol(V) =2$,  any 2-plane passing through two of these points is non-admissible.  
    \end{remark}

By some linear change of coordinates, one may assume that the hyperplane $\cH$ has equation 
$x_{n}=0$. We consider it as the hyperplane at infinity  for the coordinate system on $\bC^{n} = \bP^{n}\m \cH$, and  we identify $\cH$ with $\bP^{n-1}$. 
We then consider, like before, the polynomial:
\[ P_{\cH}: \bC^{n} \to \bC, \mbox{ \ }  P_{\cH}(x_{0}, \ldots , x_{n-1}) := f(x_{0}, \ldots , x_{n-1}, 1).
\]

Let us remark that if our $\cH$ is admissible then $\deg P_{\cH} = \deg f = d$. Indeed,   $\deg P_{\cH}<d$ is equivalent to $x_{n}$ being a factor of $f$, and therefore $V$ contains $\cH$, which contradicts \eqref{eq:star}. 

\medskip

Recalling the Definition \ref{d:T0type} of  $\cT_{0}$-type polynomials, and Definition \ref{d:t-sing} of isolated $t$-singularities at infinity, we now can show:


\begin{corollary}\label{l:isolated}
Let $V := \{f=0\}\subset \bP^{n}$. If  the hyperplane  $\cH = \{ x_{n}=0\}$ is admissible for $f$,  then the polynomial
  $P_{\cH}$ is of $\cT_{0}$-type,  i.e. outside $\overline{P_{\cH}^{-1}(0)}$ it has only isolated  $t$-singularities at infinity and only isolated affine singularities. 
\end{corollary}
\begin{proof}
Since $\cH$ is admissible and thus verifies condition \eqref{eq:star}, Theorem \ref{t:finiteThom} applies and shows that the polynomial $P_{\cH}$,   outside the fibre $\overline{P_{\cH}^{-1}(0)}$, has $t$-singularities at infinity at only finitely many points  on $\bX^{\ity} =(V\cap \cH)\times \bC$, in the notations at the beginning of \S\ref{ss:singinfty}.

The affine singularities of $P_{\cH}$ outside the fibre $P_{\cH}^{-1}(0)$ are precisely the set $\Gamma(x_{n}, f) \cap \{x_{n}=1\}$.   By the admissibility condition (ii) of Definition \ref{d:admissible}, the polar locus $\Gamma(x_{n}, f)\subset \bC^{n+1}$ is a homogeneous curve (thus a collection of finitely many lines through $0\in \bC^{n+1}$) or it is empty,
   and therefore its intersection with $\{x_{n}=1\}$ is of dimension $\le 0$. This proves the property (a) of Definition \ref{d:T0type}.  
   
    Now, since the affine singularities of the polynomial $P_{\cH}$ outside $P_{\cH}^{-1}(0)$ are isolated, it follows that the $t$-singularities at infinity  of  $P_{\cH}$ outside $V=\overline{P_{\cH}^{-1}(0)}$ are also isolated in the sense of Definition \ref{d:t-sing}.
\end{proof}


\subsection{The polar intersection multiplicities at isolated $t$-singularities at infinity.}\label{ss:jumpstinfinity} \
 
 If $P$ is a polynomial of $\cT_{0}$-type then,  outside $\bX_{0} = V$, the map $\tau$ has only isolated singularities of two types: affine singular points, and $t$-singularities at infinity, both sets being finite.  Let us denote by $\Sing^\ity P$ the set of $t$-singularities of $P$ at infinity. The image $\cB_P := \tau(\Sing^\ity P \cup \Sing P) \subset \bC$ is a finite set by our assumption. It is actually finite for any polynomial $P$, without any assumption, see \cite{Ti-compo, Ti-book}.
 
 Let $\Gamma_{(p,t)} (x_{n}, \tau)$ be the polar curve of the map germ $ (x_{n}, \tau) : (\bX, (p,t)) \to \bC \times \bC$ in some affine chart $U\subset \bP^{n}$ at the point $p$.  
   It has been proved in \cite[Theorem 2.1.7]{Ti-book}
 that in case the $t$-singularities at infinity are isolated, we have the equivalence:
 
\begin{equation}\label{eq:lambda(p,t)}
  \Gamma_{(p,t)} (x_{n}, \tau) \not= \emptyset \Longleftrightarrow  (p,t)   \mbox{ is an isolated $t$-singularity at infinity of }P.
\end{equation}

  In our setting of $\cT_{0}$-type polynomials, we apply this to $t\not=0$.

\begin{definition}\cite[Corollary 3.3.1]{Ti-book}\label{d:lambdainfty}
  Let $(p,t)\in \Sing^\ity P$  be an isolated $t$-singularity at infinity. The local intersection multiplicity  $\int_{(p,t)}(\Gamma (x_{n}, \tau), \bX_{t})$ does not depend on the choice of the chart $U$. We then call
 \begin{equation}\label{eq:mu-int}
   \lambda(p,t)= \int_{(p,t)}(\Gamma (x_{n}, \tau), \bX_{t})
\end{equation}
the polar intersection multiplicity at $(p,t)$.
\end{definition}
\sloppy
\begin{remark}\label{r:milnor-lenumber}
We shall see from the proof of the next theorem that $\lambda(p,t)$ represents the number of  vanishing cycles of the general fibre of the polynomial map $P$ which disappear at the point $(p,t)$ when this fibre converges to the fibre ${P}^{-1}(t)$. In the particular case 
 when our $\partial$-Thom stratification at infinity is a  Whitney stratification and $(p,t)$ is an isolated singularity of the function $\tau : \bX \to \bC$, then it was shown in  \cite[Proposition 3.3.5]{Ti-book} that $\lambda(p,t)$ equals the Milnor-L\^{e} number\footnote{The  ``Milnor-L\^{e} fibration'' and its ``Milnor-L\^{e} number'' have been introduced by L\^{e} D.T. in the setting of holomorphic functions  with isolated singularity on a Whitney stratified hypersurface germ, cf \cite{Le}, extending Milnor's fibration on a smooth space. It was used in \cite{ST-singinf} (and more other papers) as a measure of vanishing cycles at infinity.} of the function $\tau$ at $(p,t)$.    
\end{remark}

It  thus follows from \eqref{eq:lambda(p,t)} and \eqref{eq:mu-int} that:

\begin{equation}\label{eq:lambda(p,t)2}
  (p,t)   \mbox{ is an isolated $t$-singularity at infinity of } P  \Longleftrightarrow  \lambda(p,t) >0.
\end{equation}
 
\bigskip

With all the above notations and preparations,  recalling in particular $\cB_P := \tau(\Sing^\ity P  \cup \Sing P) \subset \bC$,  the main result of this section is: 

\begin{theorem} \label{p:F-concentration}
Let $P : \bC^n \to \bC$, $n\ge 2$,  be a non-constant $\cT_{0}$-type polynomial.

Then, for  any disk $D_{0}$ such that $\cB_P  \cap D_{0} = \{ 0\}$,   the relative $\bZ$-homology of the tube $P^{-1}(D_{0})$ is concentrated in dimension $n-1$.
Its  top Betti number  $b_{n-1}(P^{-1}(D_{0}))$  is equal to the sum of all Milnor numbers of the isolated affine singularities outside the fibre $P^{-1}(0)$, together with the sum of the polar multiplicities $\lambda(p,t)$ at the isolated $t$-singularities at infinity outside the fibre $\bX_{0}$. 
 \end{theorem}

\begin{proof}
By \cite[Corollaries 1.2.13, 1.2.14]{Ti-book}, the set  $\cB_P$ includes the set of atypical values $\Atyp P$, which means that
 one has a locally trivial topological fibration:
\begin{equation}\label{eq:toptriv1}
P_{|} : \bC^{n}\m P^{-1}(\cB_P) \to \bC\m \cB_P.
\end{equation}

Therefore, if $b\not\in \cB_P$, then the fibre $F_b := P^{-1}(b)$
is nonsingular and has no $t$-singularities at infinity. Let $D_b$ denote a closed disk centred at $b\in \cB_P$, small enough   such that $D_b\cap\cB_P = \{b\}$.   Let  $t_b \in \partial D_b$ denote a fixed point on the boundary of the disk.  We employ the notation $F_{K}:= P^{-1}(K)$, for any $K\subset \bC$.
By using  excision in the locally trivial fibration  \eqref{eq:toptriv1} we get the homology decomposition:
\begin{equation}\label{eq:direct_sum}
H_*(\bC^n , F_{D_{0}}) \simeq \oplus_{b\in \cB_P\m \{0\}} H_*(F_{D_b}, F_{t_b}).
\end{equation} 
 By Step 1 of the proof of \cite[Bouquet Theorem 3.2.1]{Ti-book}\footnote{A similar bouquet statement was shown  by Parusi\' nski  \cite[Theorem 2.1]{Pa}, with a different proof.}
under the hypothesis ``the fibre $F_{b}$ has isolated affine singularities and isolated $t$-singularities''
the relative homology $H_*(F_{D_b}, F_{t_b})$ is concentrated in dimension $n$ essentially because it is localizable at the isolated singularities of the fibre $\bX_{b}$. 
More precisely, denoting by $B_{\e}(q)$ some small enough ball at $q$, the relative homology  $H_*(F_{D_b}, F_{t_b})$ is
  the direct sum of the relative homology groups $H_*(B_{\e}(q)\cap F_{D_b}, B_{\e}(q)\cap F_{t_b})$ for all singular 
  points $q \in F_{b}$ and all $t$-singularities $q  \in \bX_{b}\cap H^{\ity}$, as done in the proof of Step 1 of \cite[Theorem 3.2.1]{Ti-book}, notably by \cite[\S 3.2, (3.16), (3.17), (3.21)]{Ti-book}. This shows that the relative homology $H_*(B_{\e}(q)\cap F_{D_b}, B_{\e}(q)\cap F_{t_b})$ is concentrated  in dimension $n$.
     In \emph{loc.cit.} the relative $n$-cycles are called the ``vanishing cycles at infinity''  at the point $(q, b)$ whenever $q  \in \bX_{b}\cap \cH$, and it is shown in \cite[\S 3.3, (3.22), (3.23)]{Ti-book} that the number of vanishing cycles is  equal to the polar intersection multiplicity $\lambda(p,t)$ defined at  \eqref{eq:mu-int}. 

It then follows  from \eqref{eq:direct_sum} that 
the relative homology $H_*(\bC^n , F_{D_{0}})$ is concentrated in dimension $n$, and so the reduced homology $\tilde H_*( F_{D_{0}})$ is concentrated in dimension $n-1$.
\end{proof}

\begin{remark}\label{r:euler}
The above result can be regarded as a far reaching global version of a local statement  known under the name of ``special fibre theorem'', proved in  \cite{Si-specialfibre}.

 In contrast to the local setting, the tube $P^{-1}(D_{0})$ is not necessarily contractible to its central fibre $V \m H^{\ity}$, and the equality $b_{n-1}(F_0) = b_{n-1}(F_{D_0})$  is also not true in general. This can be seen  in the simple example $f(x,y) = x+ x^{2}y$.  
 Nevertheless the Euler characteristics of the tube $P^{-1}(D_{0})$ and of its central fibre $V \m H^{\ity}$ are the same.
 Indeed we have $\chi(P^{-1}(D_{0})) = \chi (P^{-1}(D_{0}) \m V) + \chi (V \m H^{\ity})$. Then $P^{-1}(D_{0}) \m V$ retracts to $P^{-1}(\partial D_{0})$  due to the condition $\cB_{P}\cap D_{0} = \emptyset$, which is the total space of a locally trivial over a circle and hence its Euler characteristic is 0.
\end{remark}


The above proof may be extended at the homotopy type level. We then obtain the following result, which will not be used here but we state it for the record. This extends  \cite[Theorem 3.2.1]{Ti-book}, and in particular the bouquet theorem of \cite{ST-singinf}:

\begin{theorem}
If $P : \bC^n \to \bC$, $n\ge 2$ is a non-constant $\cT_{0}$-type polynomial, then the tube $F_{D_{0}}$ is homotopy equivalent to a bouquet of $n-1$ spheres, for  any disk $D_{0}$ such that $\cB_P  \cap D_{0} = \{ 0\}$.
The number of these spheres is equal to the number $\beta(V,\cH)$ defined  at \eqref{eq:betanumber}.
\fin
\end{theorem}

\medskip
\section{Quantization of the polar degree}\label{s:quantisation}

In this section we describe the quantifiers of the polar degree $\pol(V)$ and we prove the main decomposition result for $\pol(V)$.  Firstly, let us recall a very well known result due to L\^{e} D.T. \cite{Le} that a function with a stratified isolated singularity on a singular hypersurface of dimension $n-1$ has a local Milnor fibre (also called Milnor-L\^{e} fibre) the homotopy type of which is a bouquet of $(n-2)$-dimensional spheres, where $n\ge 3$, and that the number of these spheres is called Milnor-L\^{e} number.

\subsection{The local Milnor-L\^e number $\alpha_p(V, \cH)$.}\label{ss:sectionalM} \

In some affine chart  $\bC^{n}\subset \bP^{n}$ containing $p\in V$, let us consider a  linear function  $l: \bC^n \to \bC$  such that $l(p) =0$, and let $H_s := \{l=s\}$ for $s\in \bC$, where $H_{0} := H$.  In particular,  the projective closure of $H$ is a hyperplane  $\cH\in \bP^{n}$ which contains the point $p$.
 We  assume that  $\cH$ is transversal to all the strata of the stratification $\bP \cW$ of $V$ in the neighborhood of $p$, except 
 at the point $p$ itself. This is equivalent to saying that the restriction of  the function $l$ to some small neighborhood $B_{\e}$ of $p$ in $\bP^{n}$  has a stratified isolated singularity at $p$ with respect to $\bP \cW$. Consequently, its local Milnor-L\^e fibre  $B_{\e}\cap (V \cap H_s)$, for some $s$ close enough to $0$, has the homotopy type of a bouquet of spheres of dimension $n-2$, cf L\^{e}'s  results in \cite{Le}. 
We denote its Milnor-L\^e number by $\alpha_p(V, \cH)$, and we remark that if $\cH_{\gen}$ is a general hyperplane through $p$, then  $\alpha_p(V,\cH_{\gen})$ is the Milnor number of the \emph{complex link} of  $V$ at $p$, and thus we shall denote 
 it simply by $\alpha_p(V)$.

\begin{remark}\label{r:reducedstructure}
By its definition,  the integer $\alpha_p(V, \cH)$ is non-negative, and depends only on the \emph{reduced} structure of $V$ at $p$, and on the chosen hyperplane $\cH$. We  send to Proposition \ref{p:specialpoints} for the proof that  $\alpha_p(V) \not= 0$ only for finitely many points $p\in V$.

Let $f_{p} =0$  be a local equation of the reduced hypersurface germ $(V,p)$. Then $\alpha_p(V,\cH)$
equals the \emph{polar multiplicity of $f_{p}$ with respect to $l$ at $p$}:
\begin{equation}\label{eq:polarMilnor}
 \alpha_p(V,\cH) =  \int_{p}(H_{0}, \Gamma (l,f_{p})), 
\end{equation}
\sloppy
This polar intersection multiplicity might be  higher than the \emph{generic polar number} $\mult_{p}\Gamma (l_{\gen},f_{p})$, which is equal to $\alpha_p(V)$.

We refer to \cite{Ti-tang} for more details and for other results on the Milnor number of the hyperplane slices to  complex analytic space germs.  
\end{remark}

 \subsection{The impact of admissible hyperplanes.}
 Let  now $\cH$ be an admissible hyperplane for $V$ (cf Definition \ref{d:admissible}). We  define:
 
\begin{equation}\label{eq:alphanumber} 
\alpha(V,\cH) := \sum_{p\in V\cap \cH} \alpha_p(V,\cH)
\end{equation}
as  the sum of  the sectional Milnor numbers $\alpha_p(V,\cH)$ that have been defined at \eqref{eq:polarMilnor}. Let us show that $\alpha(V,\cH)$ is a well-defined non-negative integer.
\begin{lemma}
One has  $\alpha_p(V,\cH) > 0$ only if $p$ belongs to the finite subset $\{p_{i}\}_{i=0}^{r}\subset V\cap \cH$ of non-transversality from condition \eqref{eq:star}.   In particular, $\alpha(V,\cH) = \sum_{i=0}^{r} \alpha_{p_{i}}(V,\cH)$.
\end{lemma}
\begin{proof}
  Our  claim is equivalent to the following:  the polar locus germ $\Gamma_{p}(l, f_{p})$ is empty 
 if  $p\not\in \{p_{i}\}_{i=0}^{r}$, and it is of dimension $\le 1$ if $p\in \{p_{i}\}_{i=0}^{r}$.   
  
Indeed, if $p$ is a point of stratified transversal  intersection  $\cH \pitchfork_{p} V$, then the polar locus $\Gamma_{p}(l, f_{p})$ is empty as a direct consequence of its definition.  If $p\in \{p_{i}\}_{i=0}^{r}$,  let $B_{p}$ denote a small enough ball centred at $p$.
 By the condition of  isolated non-transversality \eqref{eq:star},  the polar locus $\Gamma_{p}(l, f_{p})$ intersects $B_{p}\cap V$ at most at $p$, thus $\dim \Gamma_{p}(l, f_{p})\le 1$. 
\end{proof}

 The following non-negative integer is also well-defined since each sum is finite, by  Corollary \ref{l:isolated}:
\begin{equation}\label{eq:betanumber}
 \beta(V,\cH)  := \beta^\aff(V,\cH)  + \beta^{\infty}(V,\cH), 
\end{equation}
where  
\[\beta^\aff(V,\cH) :=  \sum_{v\in (\Sing P_{\cH}) \m V} \mu_{v}( P_{\cH})\]
 is the sum of all Milnor numbers of the isolated affine singularities outside the fibre $P_{\cH}^{-1}(0)\subset V$, and 
 \[\beta^{\infty}(V,\cH) := \sum_{t\not= 0,  q\in V\cap \cH} \lambda(q,t) \]
  denotes the sum of the  
 numbers of vanishing cycles at infinity $\lambda(q,t)$ of the isolated $t$-singularities at infinity outside the fibre $\bX_{0} = V$,  cf Definition \ref{d:lambdainfty}. 
 
 \begin{remark}
Under the admissibility condition, the number $\beta(V,\cH)$ depends only on the \emph{reduced} structure of $V$, and on the chosen hyperplane $\cH$. Indeed, by Remark \ref{r:reducedstructure} the integer $\alpha(V,\cH)$ depends only on the reduced structure of $V$, and on the chosen hyperplane $\cH$. We then get the claimed independence of $\beta(V,\cH)$ by the equality
\eqref{eq:alpha} of Theorem \ref{t:polformula}, since we already know that $\pol(V)$ depends on the reduced structure of $V$ only.

According to the formula \eqref{eq:grad} for $\pol(V)$, the number $\beta^\aff(V,\cH)$ may also be viewed as a \emph{singular polar degree}, i.e.  
$$\pol_{\cH}(V) :=  \# (\grad f)^{-1}(\hat l) =  \mult_{0} \Gamma(\hat l, f) ,$$
 for $\hat l \in \check \bP^{n}$ defining the admissible $\cH$.

\end{remark}
  
 With these notations, we may now give the following presentation of $\pol(V)$ as a sum of \emph{local non-negative invariants} for hypersurfaces $V$ with any singular locus, extending Huh's result  for $V$ with isolated singularities  \cite[Theorem 2]{Huh}. 

\begin{theorem} \label{t:polformula}
 Let $V  := \{f=0\} \subset \bP^{n}$ be a projective hypersurface and let  
  $\cH$ be an admissible  hyperplane for $V$.  Then:
\begin{equation}\label{eq:alpha}
\pol(V)  = \alpha (V,\cH)   +  \beta (V,\cH).
\end{equation}      
\end{theorem}

\

\begin{proof}

By Theorem \ref{p:F-concentration},  which applies here via  Corollary \ref{l:isolated}, the relative homology $H_*(\bC^{n}, P_{\cH}^{-1}(D_{0}) )$ is concentrated in dimension $n$. Moreover, the
top Betti number $b_{n-1}(P_{\cH}^{-1}(D_{0})) = b_{n}(\bC^{n}, P_{\cH}^{-1}(D_{0}) )$ is a sum of certain Milnor numbers and intersections numbers  at infinity,  where $D_0\subset \bC$ is some disk such that $\cB_{P_{\cH}}  \cap D_{0} = \{ 0\}$. 

 This is by definition our number  $\beta (V,\cH)$.
 
 \smallskip

Since the tube $P_{\cH}^{-1}(D_{0})$ and the fibre $P_{\cH}^{-1}(0)$ have the same Euler characteristic (Remark \ref{r:euler}),  we get: 
\[  (-1)^{n} \rank H_{n}(\bC^{n}, P_{\cH}^{-1}(D_{0}) ) =    \chi (\bC^{n}, P_{\cH}^{-1}(D_{0}) )=  1- \chi (P_{\cH}^{-1}(D_{0}) ) = 1 - \chi(V \m \cH).
\]
Consider  the germ of a  pencil $\cP_\delta$ of hyperplanes  of $\bP^{n}$  parametrized by an arbitrarily small disk $\delta\subset \bC \subset \bP^{1}$ centred at 0,
where $\pi : \cP_\delta \m A \to \delta$ is the projection to the parameter, which contains our admissible hyperplane $\cH$ and such that 
$\pi(\cH) = 0$.
We require that $\cP_\delta$ is generic with respect to $V$, in the sense that the axis $A$ of this pencil $\cP_\delta$ (which is of dimension $n-2$) is transversal to the Whitney stratification $\bP \cW$ of  $V\subset \bP^{n}$, and more precisely transversal to the induced stratification  $\bP \cW_{H}$ on the slice $V\cap \cH$.  The choice of the axis covers a Zariski-open subset of all hyperplane slices of $V\cap \cH$.

 Since the  general member $\cH_{\gen}$ of this pencil germ is a general hyperplane with respect to $V$,  by \eqref{eq:pollink} we have:
\[   (-1)^{n} \pol(V)= 1 - \chi(V \m \cH_{\gen}).
\]
Taking the difference, we obtain:
\[ \begin{array}{l}
(-1)^{n} [ \pol(V) - \rank H_{n}(\bC^{n}, P_{\cH}^{-1}(D_{0}) ) ] = \\
\hspace*{2cm}
= \chi(V \m \cH) - \chi(V \m \cH_{\gen})  = \chi(V \cap \cH_{\gen}) - \chi(V \cap \cH). 
\end{array}\]


\medskip
Since the axis $A$ of the pencil $\cP_\delta$  is stratified-transversal to $\bP \cW$ and the stratified singularities of the pencil $\cP_\delta$ outside $A$ are precisely  the set of points of non-transversality $\Sing_{\bP \cW} (V\cap \cH)$ of \eqref{eq:star}, it follows that
 the variation of the topology of the pencil $\cP_\delta$  at  its fibre $\cH$ is localizable, by excision, at the points $q \in \Sing_{\bP \cW} (V \cap \cH)$, the pencil being equisingular at the other points of $V\cap \cH$.
 
  In homology, this variation is concentrated in dimension $n-1$, and its contribution is the number  
 $\alpha_q(V,\cH)$ defined at \eqref{ss:sectionalM}.

More precisely, for some small enough ball $B_{q}$ at $q\in V\cap \cH$, we have the following trivial, respectively locally trivial, 
fibrations induced by pencil projection $\pi$, where $\cP_\delta^{*} := \pi^{-1}(\delta\m \{0\})$, and where the radius of $\delta$ is much smaller than the radii of the balls $B_{q}$ such that they are Milnor data of the function $\pi$ at each point $q$:
\[   V\cap (\cP_\delta \m A) \m \bigsqcup_{q \in  \Sing_{\bP \cW}(V \cap \cH)} B_{q} \longrightarrow \delta  \ \ \mbox{ and } \ \    B_{q} \cap V\cap \cP_\delta^{*} \longrightarrow \delta\m \{0\} 
\]
Since $B_{q} \cap V\cap \cP_{\delta}$ is contractible for all $q\in \Sing_{\bP \cW}(V \cap \cH)$, and since  $V\cap \cP_\delta$ contracts to $V \cap \cH$, 
we get, by excising the complement of $\bigsqcup_{q}B_{q}$ in the pair $(V\cap \cP_\delta , V \cap  \cH_{\gen})$:

\[ 
\begin{split}
   \chi(V \cap  \cH_{\gen}) - \chi(V \cap \cH)
 = - \sum_{q \in  \Sing_{\bP \cW}(V \cap \cH)} \chi(B_{q} \cap V \cap \cP_{\delta}, B_{q} \cap V \cap \cH_{\gen}) \\
 = (-1)^{n}\sum_{q \in  \Sing_{\bP \cW}(V \cap \cH)} \alpha_q(V,\cH)
\end{split}
\]
From this we obtain:
\[ \beta (V,\cH) = \rank H_{n}(\bC^{n}, P_{\cH}^{-1}(D_{0}) ) = \pol(V) - \alpha (V,\cH).
\]
This ends the proof of our formula.
\end{proof}

\medskip


One may derive lower bounds for the polar degree from \eqref{eq:alpha} either from the contribution of $\alpha(V,\cH)$ or  from that of $\beta(V,\cH)$. We refer to \S\ref{ss:computpolar} and \S\ref{ss:lower1}  for more details on lower bounds, and we give below some examples with explicit computations of $\alpha(V,\cH)$ and $\beta(V,\cH)$.

 
\medskip



 \subsection{On the effective computability of the numbers $\alpha(V,\cH)$ and $\beta(V,\cH)$.}\label{ss:computpolar}\
 
 Our formula \eqref{eq:alpha} compared  to Huh's \eqref{eq:huhkey},  and to Dimca-Papadima's \eqref{eq:red}, 
amounts to a quantization of the relative homology group  $H_{n}(\bP^{n} \m V, (\bP^{n} \m V)\cap \cH)$ into localized numerical invariants depending on the admissible hyperplane $\cH$. Each local invariant is the number of solutions of a system of algebraic equations, and therefore computable, with help of adequate software, as follows. 
 
 The  local Milnor-L\^{e} numbers $\alpha_p(V,\cH)$ which compose the term $\alpha(V,\cH)$ of \eqref{eq:alpha}  are computable by the formula \eqref{eq:polarMilnor} as local polar multiplicities at each point $p\in V\cap \cH$. The more recent paper \cite{ST-polar1} shows how the  numbers $\alpha_p(V,\cH)$ occur in the particular case of $\dim \Sing V =1$.

The number $\beta^\aff(V,\cH)$ represents the total Milnor number of the polynomial $P_{\cH}$ outside the fibre $P_{\cH}^{-1}(0)$, and is computable algebraically.  
  
\smallskip

The number $\beta^\ity(V,\cH)$ is the sum of polar intersection multiplicities at points at infinity, which are computable by \eqref{eq:mu-int}.

 Moreover, due to the positivity characterisation \eqref{eq:lambda(p,t)2}, and by the formula \eqref{eq:alpha}, the number of $t$-singularities at infinity of a $\cT_{0}$-type polynomial outside its fibre over 0  provides a lower bound for $\pol(V)$.
In order to detect the set of $t$-singularities at infinity of a $\cT_{0}$-type polynomial $P$ outside $P^{-1}(0)$ by using the characterisation \eqref{eq:lambda(p,t)},  we may do as follows. According to \cite[Prop. 2.1.3]{Ti-book}, in any chart $U_{i} := \{x_{i}\not= 0\}$ we consider the affine polar locus $\Gamma_{(p,t)} (x_{n}, \tau)^{(i)}$, and then we have for it the equality of polar germs at the point $(p,t)\in \bX^{\ity}$ (in the notations of \S\ref{ss:jumpstinfinity}):
\[ \Gamma_{(p,t)} (x_{n}, \tau)^{(i)} = \overline{\Gamma}_{(p,t)}(x_{i}, P)\]
The affine polar curve $ \overline{\Gamma}_{(p,t)}(x_{i}, P)$ is defined by a number of algebraic equations, and  the asymptotic values $t\not= 0$ as well as the points $(p,t)$ can be found effectively, as done in \cite{JT}, \cite{DTT}. 
 
These methods of computation are used in the examples displayed below.

 \medskip

\subsection{Contributions of types $\alpha(V,\cH)$ and $\beta(V,\cH)$.}\label{ex:betacontrib}

\begin{example}[Hypersurfaces with isolated singularities]
Our first example is the hypersurface 
$$V:= \{ x (xy+z^{2}) = 0\} \subset \bP^{2}.$$
We compute its polar degree by using the formula \eqref{eq:alpha} of Theorem \ref{t:polformula}. It has a single isolated  singularity at the point $[0;1;0]$, which is in particular the single \emph{special point}\footnote{See the next section for the definition of special points.} of $V$.

Let $\cH := \{z=0\} \subset \bP^{2}$. Then $V\m \cH = P_{\cH}^{-1}(0)$ for $P_{\cH}(x,y) = x (xy+1)$. This polynomial has an interesting behaviour at infinity,  well-known in the literature (e.g. \cite{Br}, \cite{ST-singinf}), and therefore we will extract here some results without giving the full details.  The polynomial $P_{\cH}$ has no affine critical points, hence  $\beta^\aff(V,\cH) =0$. At infinity, $P_{\cH}$ may have $t$-singularities only on the line $\{[0;1;0]\}\times \bC$, and they are  isolated. It turns out that  the unique isolated $t$-singularity at infinity of $P$ is  the point $\{[0;1;0]\}\times \{0\}$.
Since this is on $\bX_{0}$, it does not count in our computation of vanishing cycles  (cf Theorem \ref{p:F-concentration}), and therefore  $\beta^{\infty}(V,\cH)=0$. 

On the other hand, by computing the local polar curve at the special point $p= [0:1:0] \in V$ we get  $\alpha_{p}(V,\cH) = \alpha_{p}(V)=1$. By applying Theorem \ref{t:polformula} we get $\pol (V) = 1$.  
 
 \smallskip
Let us  now  consider  the modified projective curve 
$$V':= \{ x (xy+z^{2})  - bz^{3}= 0\} \subset \bP^{2}$$
 for $b \ne 0$, having a unique singularity at $p=[0;1;0]$.
Then $V'\m \cH$ is the fibre over 0 of  the polynomial $Q_{\cH}(x,y) = x (xy+1) - b$,  and $Q_{\cH}$
has still no affine critical points (thus  $\beta^\aff(V',\cH) =0$). But this time $Q_{\cH}$ has its single isolated $t$-singularity at infinity at the point $\{[0;1;0]\}\times \{b\}$ which is not on $\overline{Q_{\cH}^{-1}(0)}$, and  we obtain $\beta^{\infty}(V',\cH) =1$. This can actually be computed (see  \S\ref{ss:singinfty} for the notations, and \cite{ST-singinf}, \cite{Ti-book} for the details) as the jump $A_{2} \mapsto A_{3}$ of the local type of the singularity of the fibre $\bX_{t}$ at the point $[0;1;0] \in \bX^\ity$ at the value $t=b$. 

As above, we also get $\alpha_{p}(V',\cH) =1$ from the unique special point  $p\in V'$. By applying Theorem \ref{t:polformula} we now have  $\pol (V') = 1+1 =2$. 
\end{example}
\smallskip

\begin{example}[Non-isolated singularities.]\label{ex:alphacontrib}

We consider the projective surface 
$$V :=  \{f=  x^{2}z + xyw +y^{3} = 0\} \subset \bP^{3}$$ 
with singular locus the line $L:= \{ x=y=0\}$.  This is mentioned in \cite[\S 3.2]{CRS} as the rational scroll $Y(1,2)$, and it is related to a sub-Hankel surface
\cite[Example 4.7]{CRS}.  It is known to be a  homaloidal surface. Our approach to compute $\pol(V)=1$  goes as follows.

The coordinates of $\bP^{3}$ are $[x;y;z;w]$
and we consider the hyperplane $\cH := \{ \hat l = w = 0\}$. 
Let us first remark that $\cH$ is admissible since it has a tangency at the point $p :=  [0;0;1;0] \in L$ only, and since the polar locus $\Gamma(\hat l, f)$ is empty.  We may therefore apply Theorem \ref{t:polformula} provided that we compute its ingredients $\alpha$ and $\beta$, which we shall do in the following.

We first compute $\alpha(\cH, V)$. The line $L$ has a generic transversal type $A_{1}$, which actually holds at all points, except of the point $p=  [0;0;1;0] \in L$ which is a non-isolated singularity of type $J_{2, \infty}$.
This point $p$ is therefore  the only candidate as a \emph{special point} or $V$  (see the next section for the definition), so let us check if $p$ is special or not. 
In the chart $\{ z =1\}$, we have  $f_{p} := f_{\{ z =1\}} = x^{2} + xyw +y^{3}$,  and $H_{0}$ denotes a local representative of the projective hyperplane $\cH$. The polar locus $\Gamma(w, f_{p})$ is a curve, 
  and by solving equations we get its parametric presentation as: $(-\frac1{12}w^{3}, \frac1{6}w^{2}, w)\in \bC^{3}$ for $w\in \bC$. The intersection multiplicity at the origin  $\int_{0}(H_{0}, \Gamma(w, f_{p})$, which by definition is $\alpha_{p}(\cH, V)$,  is equal to 1.  It is thus equal to $\mult_{0}\Gamma(w, f_{p})$, and also equal to the generic Milnor number $\alpha_{p}(V)$. Therefore $p$ is indeed the unique special point of $V$. 

We next show that $\beta(V,\cH) =0$. Our $\cH$ is the hyperplane ``at infinity'' and the corresponding polynomial
$P$ defined on $\bC^{3} = \bP^{3}\m \cH$ is $P_{\cH}(x,y,z) = x^{2}z + xy + y^{3}$. This polynomial has no singularities outside 
$P_{\cH}^{-1}(0)$, thus $\beta^\aff(V,\cH) =0$. Next,  by Corollary \ref{l:isolated}, $t$-singularities at infinity occur only on the line $\{p\} \times \bC \subset \bX^{\ity}$, and that they are isolated.
We then use \S\ref{ss:computpolar} to detect them. We compute $\Gamma(z, P_{\cH})$, which is a curve.  The restriction  of $P_{\cH}$ on it detects asymptotically all the values $(p,t)$ which are the $t$-singularities at infinity. In our case 
the only asymptotical value is $t=0$, but this is not outside $\overline{P_{\cH}^{-1}(0)}$. So there are no 
isolated $t$-singularities at all, and thus $\beta^\infty(V,\cH) =0$.  Theorem \ref{t:polformula} then gives $\pol(V) = 1+0 = 1$.

\smallskip

Let us consider the cubic surface:  
$$V :=  \{f = x^2 z + y^2w = 0\} \subset \bP^3$$ 
with $\pol(V) =2$. Its singular locus is the same line $L:= \{ x=y=0\}$ of transversal type $A_{1}$, now except of two $D_{\infty}$-points $p=[0;0;1;0]$ and $q= [0;0;0;1]$.  By similar computations as above, it turns out that these two points are the only \emph{special points} of $V$ (see the next section for the definition), with 
 $\alpha_p(V) = \alpha_q(V)=1$.  
 
The  above hyperplane $\{w = 0\}$ is no more admissible since its  non-transversality locus is a line. We chose now $\cH := \{w -x-y = 0\}$. In order to compute the invariants, we apply a linear change of variables such that $\hat l$ is the new variable  $w$ and the others do not move.   In this new system of coordinates we have $V := \{f = x^2 z + y^2w -xy^{2} - y^{3}= 0\}$ and $\cH := \{ \hat l = w =0\}$ with a single non-transversality point at $p= [0;0;1;0]$.

 By using \S\ref{ss:computpolar} in the chart $z=1$, we get $\alpha_{p}(V, \cH)=1$,  thus we have $\alpha(V, \cH)=1$.
 
 For detecting $\beta(V,\cH)$ we use again \S\ref{ss:computpolar}. We consider 
 the polynomial $P_{\cH}(x,y,z) = x^{2}z + y^{2}-xy^{2} - y^{3}$ on $\bC^{3} = \bP^{3}\m \cH$ and check that it has no singularities outside the fibre $P_{\cH}^{-1}(0)$, thus $\beta^{\aff}(V,\cH) =0$.
 To find the $t$-singularities at infinity outside $\overline{P_{\cH}^{-1}(0)}$, we compute the affine polar curve $\Gamma(z,P_{\cH})$. It has the parametrization $(1-\frac3{2}y,  y, \frac{y^{2}}{2-3y})\in \bC^{3}$ for $y\in \bC$. We compute the asymptotic values of the restriction $ P_{|\Gamma}$ for all branches which converge to infinity, and we find that the only finite value different from 0  is  $\lim_{y\to \frac23} P_{|\Gamma}= \frac 4{27}$. The point $(p, \frac 4{27})\in \bX^{\ity}$ is therefore the unique $t$-singularity at infinity outside $P_{\cH}^{-1}(0)$. We then compute the local intersection multiplicity $\int_{(p, \frac 4{27})} (\overline{\Gamma}_{(p,t)}(z, P_{\cH}), \bX_{\frac 4{27}})$ and find that its value is 1.
 
   This yields $\beta^{\infty}(V,\cH)=1$, hence we get $\beta(V,\cH) =1$.  Finally $\pol (V) = 1+1=2$ by Theorem \ref{t:polformula}.

\end{example}
 

\medskip


\section{Special points,  and lower bounds for the polar degree}\label{s:specV} 

\subsection{The definition of special points of $V$.}\label{ss:specialpoints}
Recalling that the local numerical invariant $\alpha_p(V)$ is defined in \S\ref{ss:sectionalM}, let us introduce the following:

\smallskip
\begin{definition}[Special points of $V$]\label{d:special}
We say that $p\in \Sing (V)$ is a {\it special point of $V$} if  $\alpha_p(V) > 0$.
\end{definition}

\begin{remark}
 The notion of ``special point'' can be defined not only for a hypersurface $V$ in $\bP^{n}$ but for any singular complex space $X$ of pure dimension $n-1$ such that the shifted constant sheaf $\underline{\bQ}_X[n-1]$ is a $\bQ$-perverse sheaf on $X$.
\end{remark}
  
 The next result tells that the set of all special points of $V$ is an intrinsic invariant.
 
\begin{proposition}\label{p:specialpoints}
 The set $V_{\spec}\subset V$ of special points is a finite subset of $V$ which is independent of the stratification of $V$.
  In particular,  $V_{\spec}$ is included in the set of point-strata of the coarsest Whitney stratification of $V$.
\end{proposition}
\begin{proof}
If the point $p$ is on a stratum of dimension $>0$ of some Whitney stratification $\bP \cW$ of $V$, then $\alpha_p(V) = 0$ since the complex link of any stratum of positive dimension is contractible\footnote{This follows from the fact that $V$, at some point $p$ on a positive dimensional stratum of $\bP \cW$, is locally equivalent, up to homeomorphisms, to a product of the stratum with a transversal hyperplane slice. See e.g. \cite{GM} for all these well-known facts.}. Therefore only point-strata of  $\bP \cW$ can be special points (according to Definition \ref{d:special}), and these point-strata are finitely many since $V$ is compact.  It also  follows that only points $p$ which are point-strata of any Whitney stratification of $V$ can be special points, thus our last assertion is proved. 
\end{proof}

\begin{remark}
  In case $V$ has nonisolated singularities, the inclusion in the statement of Proposition \ref{p:specialpoints} might be strict. As a local example, consider  the Brian\c con-Speder \cite{BS} family  of hypersurface germs which is $\mu$-constant but not $\mu^{*}$-constant\footnote{The notation $\mu^{*}$ is due to Teissier, and ``$\mu^{*}$-constant'' means that all Milnor numbers of  the linear sections of all dimensions through some fixed point (here the point is the origin) are constant in the family.}; then the origin is a point-stratum in the  coarsest local Whitney stratification, but it is not a special point. 
\end{remark}

 In case $(V,p)$ is an (at most) \emph{isolated} singularity germ, we have $\alpha_p(V) =\mu_p(V \cap \cH_{\gen}) \ge 0$ , with annulation $\mu_p(V \cap \cH_{\gen}) = 0$ if and only if  $(V,p)$ is nonsingular\footnote{This is a well-known fact, see e.g.  \cite[Prop. 4.1]{Ti-tang} and more references therein.}. Therefore we get:
 \begin{corollary}\label{c:specisol}
 If $V$ has only isolated singularities, then $V_{\spec} = \Sing(V)$.
 \fin
\end{corollary}

%


\subsection{On lower bounds for the polar degree.}\label{ss:lower1}\

We have seen in Remark \ref{c:Zar} that the set of admissible hyperplanes for $f$ at some point $p\in \Sing V$ is a Zariski-open subset of the set of all hyperplanes through $p$. Therefore, if  $\cH$ is admissible for $f$ at $p\in \Sing V$, then we obtain from Theorem  \ref{t:polformula}:
\begin{equation}\label{eq:alphabound}
 \pol(V)  \ge \alpha_p(V,\cH) + \beta(V,\cH) \ge \alpha_p(V,\cH) \ge \alpha_p(V)  \ge 0
\end{equation}
 where the last inequality reads ``$>0$'' if and only if $p$ is a special point of $V$. 
 
\medskip


The following statement is a consequence of  Theorem  \ref{t:polformula} via the inequality \eqref{eq:alphabound}:
\begin{corollary}\label{c:admissibleandspecial}
Let $V\subset \bP^{n}$ be a projective hypersurface which is not a cone.  Then:

\begin{equation}\label{eq:newbounds}
 \pol(V)  \ge \max_{p\in V_{\spec}}\alpha_p(V).
\end{equation}
\fin
\end{corollary}

\begin{remark}\label{r:pol=0}
Corollary \ref{c:admissibleandspecial} extends to any singular locus Huh's result \cite[Theorem 2]{Huh} that holds in the setting of $V$ with isolated singularities only.  More generally, if $V$ has non-isolated singularities and also isolated singular points, then our formula \eqref{eq:newbounds} implies that if $V$ is homaloidal, then its isolated singularities must be of type $A_{k}$, just as in Huh's setting. 

Let us show here what  Corollary \ref{c:admissibleandspecial} becomes in the particular case of $V$ with  isolated singularities only. 
A point $p\in \Sing V$ is then a special point of $V$ (by Corollary \ref{c:specisol}).  Since $\alpha_p(V) = \mu_{p}^{\langle n-2\rangle}(V)$,   Corollary  \ref{c:admissibleandspecial}  yields the inequality:
 
\begin{equation}\label{eq:huhbound2}
 \pol (V ) \ge  \max_{p\in \Sing(V)} \mu_{p}^{\langle n-2\rangle}(V)
\end{equation}
which  recovers the general formulation of Huh's result \cite[Theorem 2]{Huh},   stated in \eqref{eq:huhbound} in a particular form.
\end{remark}

\begin{remark}
 If  an admissible  hyperplane $\cH$  has an isolated stratified tangency to $V$ at $p\not\in V_{\spec}$,  then this tangency contributes to $\pol (V)$ with the positive summand $\alpha_p(V,\cH)$.   In case of a non-degenerate quadratic contact, this contribution is $1$, and in case of higher order tangency, it is  $2$ or more.
In particular  homaloidal hypersurfaces may have such isolated tangency points for admissible hyperplanes, but the orders of contact cannot be higher than 1.  
\end{remark}



\subsection{On  hypersurfaces with $\pol(V) =0$.}\label{ss:pol=0}

Let us observe that Huh's result \cite[Theorem 2]{Huh} does not apply in this case.  
Indeed, the following equivalence should be well-known: \emph{if $V$ has at most isolated singularities, then $\pol(V) =0$ if and only if $V$ is a cone}.  If $V$ has isolated singularities only and it is a cone,  then it has a unique apex $q$ which is moreover the unique singular point of $V$ (since if not so, then $V$ would have non-isolated singularities in both cases).  Precisely this case ``$V$ is a cone of apex $q$'' is excluded in Huh's result \cite[Theorem 2]{Huh}.

In the same setting of $V$ with isolated singularities only,  our formula  \eqref{eq:alphabound}  does not apply in the 
case $\pol(V) =0$ either. The reason is that  the generic hyperplanes through  the unique apex $q$ of the cone $V$ are not admissible. Indeed, if the cone $V$ is not smooth, then its apex is an isolated singular point, thus a special point.  For a generic linear $l$ form in the slice $K = \{x_{0}=1\}$ at this point $q$,  the polar locus $\Gamma(l,g)$ is of dimension 1 due to the fact the the apex $q$ is an isolated singularity of $V$.  By Lemma \ref{l:nongenpolar}, it  then follows that $\dim \Gamma(\hat l,f) = 2$. This proves our claim that that all general hyperplanes $\cH$ through $q$ are non-admissible, and thus Theorem  \ref{t:polformula} does not apply.  

\begin{corollary}\label{c:coneornotcone}
 Let $V\subset \bP^{n}$ be any projective hypersurface with $\pol(V) =0$,   but not a cone. Then $V$  has no special points.
\end{corollary}

\begin{proof}
Since $p\in \Sing V$ is not an apex of a cone,  Theorem \ref{t:mainpolar1} shows, via Remark \ref{c:Zar}, that there  is a Zariski-open set of admissible hyperplanes $\cH$ through $p$. Therefore Theorem \ref{t:polformula} applies and yields the inequality:
\begin{equation}\label{eq:pol-alphap}
  \pol(V) \ge  \alpha_{p}(V) \ge 0,
\end{equation}
 where $\alpha_{p}(V) > 0$ whenever  $p\in V_{\spec}$, by definition. But this  contradicts our assumption  $\pol(V) =0$. 
\end{proof}

We end by an interesting example of a projective hypersurface which is not a cone, but has polar degree equal to 0. It has been believed long ago (see  \cite{Hes51, Hes59}) that $\pol(V) = 0$ implies ``$V$ is a cone''. Well-known to be true for $V$ with isolated singularities, this is not true whenever $V$ has nonisolated singularities, cf  \cite{GN76}.

\begin{example}\label{ex:notcone}
The projective hypersurface 
$$V := \{f=  z_{3}^{d-1}z_{0} + z_{3}^{d-2}z_{4}z_{1} + z_{4}^{d-1}z_{2} =0\}\subset \bP^{4},$$ 
 of degree $d \ge 3$,   is a well known example of a  hypersurface with vanishing Hessian, thus $\pol(V)=0$,  which is not a cone,  see e.g.  \cite{Huh}.
 
Its singular locus is the 2-plane $\Sing V = \{ z_{3} = z_{4} =0\}\subset \bP^{4}$.  This plane contains a curve $C$ of degree $d-1$ where the transversal type changes. This curve $C$  has a single singularity at the point $p= [0;0;1;0;0]$ of cusp type.

In the affine chart $\{z_{2}=1\} \simeq \bC^{4}$,  we compute the generic affine polar curve $\Gamma(l, g)$ where $g:= f_{| \{z_{2}=1\}}$ and $l: \bC^{4}\to \bC$ is a generic linear form. We get that $\dim \Gamma(l, g)=1$ and that $\Gamma(l, g)$ intersects $C$, 
but not at the point $p$ (which represents the origin of this chart). This shows that $p$ is not a special point of $V$. 
Repeating the computation at the intersection points $\Gamma(l, g) \cap \Sing V$ by using local generic linear forms at each such point, instead of the the affine generic $l$, we find that there are no local polar curves, and therefore these points are not special points either.  Altogether this confirms  Corollary \ref{c:coneornotcone}. 

 Although the generic affine polar curve in this chart is empty at the origin, there  are relatively generic hyperplanes $\cH = \{\hat l =0\}$ containing the singular point $p$ and satisfying condition \eqref{eq:star}. But all these are \emph{not admissible} because the polar locus $\Gamma(\hat l , f)$ has dimension 2, and therefore Theorem  \ref{t:polformula} does not apply.
\end{example}





\end{document}